\documentclass[leqno,
]{amsart}

\usepackage{amssymb, amsmath, amsfonts}
\usepackage{enumerate, rotating}
\usepackage[dvipsnames]{xcolor}
\usepackage[normalem]{ulem}

\usepackage[colorlinks,citecolor=blue,urlcolor=blue]{hyperref}	

\usepackage[polish, english]{babel}  
\usepackage[T1]{fontenc}
\usepackage[cp1250]{inputenc}

\newtheorem{thm}{Theorem}[section]
\newtheorem{prop}[thm]{Proposition}
\newtheorem{cor}[thm]{Corollary}
\newtheorem{lem}[thm]{Lemma}

\newtheorem{rem}[thm]{Remark}
\theoremstyle{definition}

\numberwithin{equation}{section}

 \newcommand{\Hh}{\mathcal{H}}

 \newcommand{\Ll}{\mathcal{L}}
 \newcommand{\Mm}{\mathcal{M}}
\newcommand{\NN}{\mathbb{N}} 
 
 \newcommand{\Pp}{\mathcal{P}}
 
\newcommand{\RR}{\mathbb{R}}

\newcommand{\ZZ}{\mathbb{Z}} 

\newcommand{\8}{\infty}

\renewcommand{\(}{\left(}
\renewcommand{\)}{\right)}

\newcommand{\Rn}{{\RR^n}}

\newcommand{\mc}{\mathcal}

\newcommand{\wt}{\widetilde}

\def \PPp {\mathbf{P}}

\newcommand{\supp}{\mathrm{supp}}

\title[Hardy spaces for Fourier--Bessel expansions]{Hardy spaces for Fourier--Bessel expansions}

\author[J. Dziuba\'nski]{Jacek Dziuba\'nski}
\address{Instytut Matematyczny\\
         Uniwersytet Wroc\l awski\\
         50-384 Wroc\l aw, Pl. Grunwaldzki 2/4\\
         Poland}
\email{jacek.dziubanski@gmail.com, preisner@math.uni.wroc.pl}

\author[M. Preisner]{Marcin Preisner}

\author[L. Roncal]{Luz Roncal}
\address{Departamento de Matem\'aticas y Computaci\'on\\
         Universidad de La Rioja\\
         26004 Logro\~no, Spain}
\email{luz.roncal@unirioja.es}

\author[P. R. Stinga]{Pablo Ra\'ul Stinga}
\address{Department of Mathematics\\
         The University of Texas at Austin\\
         1 University Station\\
         C1200, Austin, TX 78712-1202, United States of America}
\email{stinga@math.utexas.edu}

\thanks{The first and second authors were  supported by Polish funds for sciences  grant DEC-2012/05/B/ST1/00672
from Narodowe Centrum Nauki and by Polish funds of sciences NCN grant 2012/05/B/ST1/00672. The third author was partially supported by grant MTM2012-36732-C03-02 from Spanish Government. The fourth author was partially supported by grant MTM2011-28149-C02-01 from Spanish Government.}



\begin{document}

\begin{abstract}
We study Hardy spaces for Fourier--Bessel expansions associated with Bessel operators on $((0,1), \, x^{2\nu+1}\, dx)$ and $((0,1), dx)$. We define Hardy spaces $H^1$ as the sets of $L^1$-functions for which their maximal functions for the corresponding Poisson semigroups belong to $L^1$. Atomic characterizations are obtained.
\end{abstract}

\maketitle


\section{Introduction}\label{sec1}

In this paper we give atomic characterizations of the Hardy spaces defined with Poisson integrals associated with two different Fourier--Bessel expansions on the unit interval $(0,1)$.

Let us briefly describe one of our motivations. For $n \ge 1$, let $B_1=\{X \in \mathbb{R}^n : |X|<1\}$ be the unit ball in $\mathbb{R}^n$.
Let us consider the Dirichlet problem for $U(X,t)$, with $(X,t)\in  B_1\times(0,\infty)$,
\begin{equation*}
\begin{cases}
U_{tt}+\Delta U=0,
\\U(X,0)=F(X),\end{cases}
\end{equation*}
where $\Delta$ is Dirichlet Laplacian in $B_1$.
When the initial datum is radial, by writing $|X|=x, F(X)=f(x)$, and by expressing the Laplacian in polar coordinates, we see that the solution $U$ is radial in $X$, so $U(X,t)=u(x,t)$, for some function $u$. Therefore, one is led to the problem in $(x,t)\in (0,1)\times(0,\infty)$
\begin{equation}
\label{eq:PoissonRadial} \begin{cases}
u_{tt}-\mathcal{L}u=0,
\\u(x,0)=f(x), \end{cases}
\end{equation}
under appropriate boundary conditions, see \cite{Folland}, where $\mathcal{L}$ is the operator
\begin{equation}\label{eq:LRedonda}
\mathcal{L}=-\frac{d^2}{dx^2}-\frac{2\nu+1}{x}\frac{d}{dx},
\end{equation}
 with the type index $\nu=n/2-1$. There is no need to consider just half-integer values of $\nu$, so we can take any general index $\nu>-1$. The operator  $\mathcal{L}$ has a basis of eigenfunctions $\phi_n^\nu$ (see Subsection \ref{ortt}) in $L^2((0,1),\mu)$, where $d\mu(x)=x^{2\nu+1}dx$, see \cite[Chapter~2]{Folland}. Therefore, by applying the Fourier method, we see that the solution of the Dirichlet problem \eqref{eq:PoissonRadial} is $u(x,t)=e^{-t\sqrt{\mathcal{L}}}f(x)=\mathcal{P}_tf(x)$, the Poisson semigroup associated to $\mathcal{L}$ applied to $f$ (see \eqref{Poissone_formula} below for the precise definition of $\mathcal{P}_t$).

In order to study the almost everywhere pointwise convergence of the solution $u$ to the initial datum $f$, for $f\in L^p((0,1),\mu)$, as $t\to0^+$, one considers the maximal operator
 \begin{equation}\label{eq:maximalNatural}
  \mathcal{M}f(x)=\sup_{t>0} |u(x,t)|=\sup_{t>0} |\mathcal{P}_t f(x)|, \quad x\in (0,1).
  \end{equation}
  It is known, see \cite{Ciaurri-Roncal-Bochner, Gilbert, NR, Nowak-Roncal-Sharp}, that such operator is bounded on $L^p((0,1),\mu)$, for $1<p<\infty$, and of weak type $(1,1)$. Then, as usual, we can get the desired almost everywhere convergence. As we just pointed out, for $f\in L^1((0,1),\mu)$, the maximal operator \eqref{eq:maximalNatural} is not, in general, in $L^1((0,1),\mu)$. This motivates us to define the Hardy space $H^1$ associated with $\mathcal{L}$, namely, the maximal subspace of $L^1((0,1),\mu)$ whose image under $\mathcal{M}$ is in $L^1((0,1),\mu)$.

A related problem to \eqref{eq:PoissonRadial} arises when one defines
$$
v(x,t):=x^{\nu+1/2}u(x,t).
$$
Then, for $g(x):=x^{\nu+1/2}f(x)$, the function $v$ solves
\begin{equation*}
\begin{cases}
v_{tt}-Lv=0,&\hbox{for}~x\in(0,1),~t>0,
\\v(x,0)=g(x),&\hbox{for}~x\in(0,1), \end{cases}
\end{equation*}
where
\begin{equation}\label{eq:LCuadrado}
L =x^{\nu+1/2}\circ \mathcal{L}\circ x^{-\nu-1/2}= -\frac{d^2}{dx^2} +\frac{\nu^2 -1/4}{x^2}.
\end{equation}
 The eigenfunctions of $L$, that are clearly related to those of $\mathcal{L}$, form now a basis of $L^2(0,1)=L^2((0,1),dx)$. The maximal operator
 \begin{equation}\label{eq:maximalLebesgue}
M g(x)=\sup_{t>0} |v(x,t)|=\sup_{t>0}|P_tg(x)|
\end{equation}
  is bounded in $L^p(0,1)$, for $1<p<\infty$ and it is of weak type $(1,1)$. See \eqref{Poissone_formula} for the definition of the Poisson semigroup $P_t$. A natural question is to characterize the Hardy space associated to $L$, that is, the  subspace of functions $g$ in $L^1(0,1)$ such that $Mg$ is in $L^1(0,1)$.

Next, we present our main results, Theorems A and B below. To this end we introduce the definitions of Hardy spaces and their corresponding atoms. In the whole paper we denote by $L^p(I, \mu)$ the $L^p$-space on an interval $I\subseteq (0,\infty)$ with respect to the measure $\mu$ given above, and by $L^p(I)$ the classical $L^p$-space related to the Lebesgue measure on $I$. Let us finally mention that although all the objects above exist for $\nu>-1$, from now on we consider only the case $\nu>-1/2$. This restriction is due to the techniques we use.

\subsection{Hardy space for the operator $\mathcal{L}$}\label{subsection L redondo}
Let $\Mm$ be the maximal function given in \eqref{eq:maximalNatural}. It is well-known that $\mc P_t$ and $\Mm$ can be applied to $L^1((0,1), \mu)$-functions, see the estimates in Lemma \ref{Lem:SharpNatural} below. We say that a function $f\in L^1((0,1), \mu)$ is in the Hardy space $H^1_{\mathcal{L}}$ when
    \begin{equation*}
    \|f\|_{H^1_{\mathcal{L}}}:=\|\mathcal{M}f\|_{L^1((0,1),\mu)} <\8.
    \end{equation*}

Now we introduce the atoms associated with $H^1_{{\Ll}}$. Denote the intervals
    $$\mathcal{I}_j=(1-2^{-j},1-2^{-j-1}], \quad \hbox{for}~ j=0,1,\ldots.$$

    \begin{figure}[here]
    \includegraphics[width=0.8\textwidth]{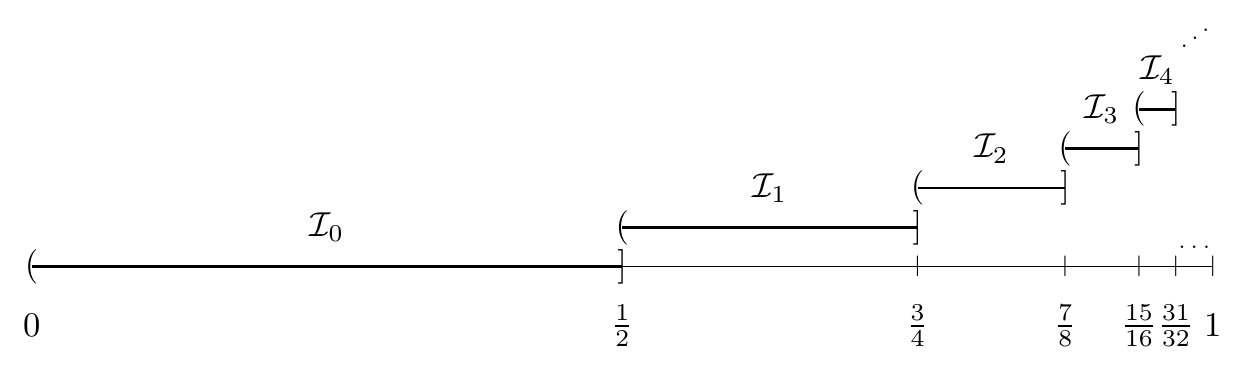}
    \caption{The intervals $\mathcal{I}_j$, for $j\geq0$.}
    \label{figureI}
    \end{figure}

Observe that $\mathcal{I}_i\cap\mathcal{I}_j=\emptyset$ for $i\neq j$ and $(0,1)=\bigcup_{j\geq0}\mathcal{I}_j.$
For an interval $I$, let $I^{*}=(1+\zeta)I\cap(0,1)$, where $\zeta=1/50$ and $(1+\zeta)I$ is the interval that has the same center as $I$, but its length is $(1+\zeta)$ times bigger.

We say that a complex valued function $a$ is an $\mathcal{L}$-atom if it either satisfies
    \begin{enumerate}
        \item[(i)] there exists an interval $I\subseteq (0,1)$ such that
            \begin{align*}
            &\supp \, a \subseteq I, \qquad \|a\|_{\8} \leq \mu(I)^{-1}, \qquad \int_I a(x)\, d\mu(x) =0,
            \end{align*}
            or
        \item[(ii)] $a(x) = \mu(\mathcal{I}_j)^{-1}\chi_{\mathcal{I}_j}(x)$, for some $j\ge0$,
    \end{enumerate}
    where $\chi_A(x)$ denotes the characteristic function of a set A.
The Hardy space $H^1_{\mathcal{L},\, \mathrm{at}}$ is given in the usual way as the set of functions $f$ that can be written as $f=\sum_{i}\lambda_ia_i$, where $\sum_i|\lambda_i|<\8$ and $a_i$ are $\mathcal{L}$-atoms. The infimum of the sums $\sum_i|\lambda_i|$ for which $f=\sum_i \lambda_i a_i$ is denoted by $\|f\|_{H^1_{\mathcal{L},\mathrm{at}}}$. The first main result of the paper is the following theorem.

\newcounter{asd}
\renewcommand{\theasd}{\Alph{asd}}
\newtheorem{thmm}[asd]{Theorem}

    \begin{thmm}[Hardy space related to $\mathcal{L}$] \label{thmA}
    Assume that $\nu > -1/2$. The Hardy spaces $H^1_{\Ll}$ and $H^1_{{\Ll},\, \mathrm{at}}$ coincide. Moreover, there exists a constant $C>1$ such that
        \begin{equation}\label{th:A}
        C^{-1}\|f\|_{H^1_{\mathcal{L}}}\le\|f\|_{H^1_{\mathcal{L},\, \mathrm{at}}}\le C\|f\|_{H^1_{\mathcal{L}}}.
        \end{equation}
    \end{thmm}

\subsection{Hardy space for the operator $L$}
For a function $g\in L^1(0,1)$ and $x\in(0,1)$, we consider the maximal operator $M$ as in \eqref{eq:maximalLebesgue}. This operator is well defined, see the estimates for the Poisson kernel in Lemma \ref{Lem:SharpLebesgue} below. The space $H^1_L$ is defined as the set of functions $g\in L^1(0,1)$ such that
    $$\|g\|_{H^1_L}:=\|Mg\|_{L^1(0,1)}<\8.$$

Let us define the atoms for this space. For $j\in \ZZ^* = \ZZ\setminus\{0\}$, set
    $$\mathcal{J}_j=\mathcal{I}_j \ \text{  when  } \ j\geq 1 \quad \text{ and } \quad \mathcal{J}_{j}=(2^{j-1},2^{j}] \ \text{  when  }\  j\leq -1.$$

    \begin{figure}[here]
    \includegraphics[width=0.8\textwidth]{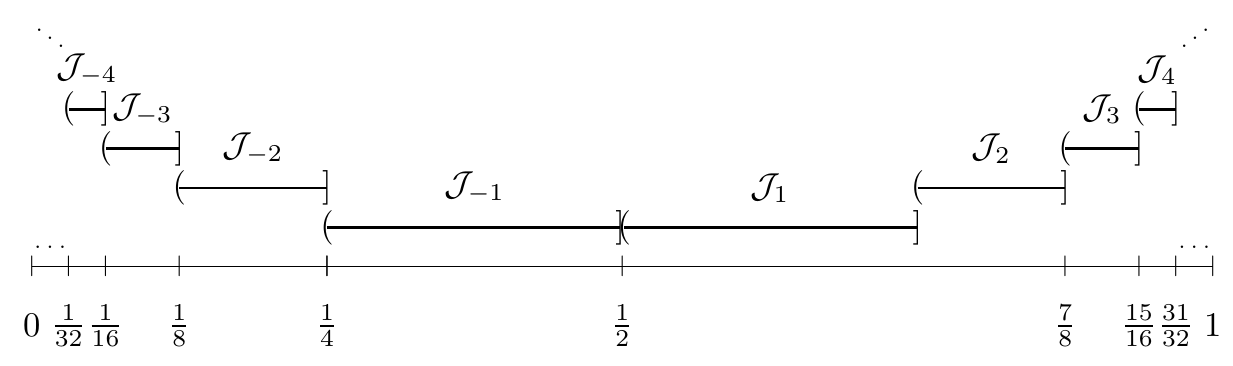}
    \caption{The intervals $\mathcal{J}_j$, for $j\in\mathbb{Z}^\ast$.}
    \label{figureJ}
    \end{figure}

\noindent
Obviously, $\mathcal{J}_i\cap\mathcal{J}_j=\emptyset$ for $i\neq j$ and $(0,1)=\bigcup_{j\in \ZZ^*}\mathcal{J}_{j}.$

A complex valued function $a$ is called an $L$-atom if  it either satisfies
    \begin{enumerate}
        \item[(i)] \label{atom1} there exists an interval $I\subseteq (0,1)$ such that
            $$\supp \, a \subseteq I, \quad \|a\|_{\8} \leq |I|^{-1},\quad  \text{and} \quad \int_I a(x)\, dx =0,$$
or
        \item[(ii)] \label{atom2} $a(x) = |\mathcal{J}_j|^{-1}\chi_{\mathcal{J}_j}(x)$, for some $j\in \ZZ^*$.
    \end{enumerate}
The atomic Hardy space $H^1_{L,\,\mathrm{at}}$ and its norm $\|\cdot\|_{H^1_{L,\,\mathrm{at}}}$ are defined as usual. The second main result of this paper is the following.

    \begin{thmm}[Hardy space related to $L$] \label{thmB}
        Assume that $\nu > -1/2$. The Hardy spaces $H^1_L$ and $H^1_{L,\,\mathrm{at}}$ coincide. Moreover, there exists a constant $C>1$ such that
$$C^{-1}\|g\|_{H^1_{L}}\le\|g\|_{H^1_{L,\,\mathrm{at}}}\le C\|g\|_{H^1_{L}}.$$
\end{thmm}

For the theory of the classical real Hardy spaces on $\Rn$ and their equivalent characterizations we refer the reader to \cite{Coifman, Fefferman-Stein, Goldberg, Latter, SteinWeiss-Acta}. See also \cite{Stein} and references therein.
Harmonic Analysis in the setting of Fourier--Bessel expansions is being developed in the last years, see \cite{Benedek-Panzone-negative, Benedek-Panzone, Ciaurri-Roncal-Bochner, Ciaurri-Roncal-g, Ciaurri-Roncal, Ciaurri-Stempak-Trans, Ciaurri-Stempak-Anal, Ciaurri-Stempak,  NR, Nowak-Roncal-Sharp,Nowak-Roncal}.

Note that our Hardy spaces are defined in terms of maximal functions given by Poisson integrals. It is natural to ask for another characterizations of $H^1_L$ and $H^1_{\Ll}$ analogous to those from the classical theory. Another interesting task is to study the Hardy spaces $H^p_{\mathcal{L}}$ and $H^p_{L}$ for $0<p<1$. We will not pursue these issues in this paper.

The ideas to prove our two main results are the following. For the Hardy space $H^1_{\mathcal{L}}$, first, we take an $\mathcal{L}$-atom $a$ and we prove that $\mathcal{M}a$ is in $L^1((0,1),\mu)$ by using precise kernel estimates and by applying a result by Uchiyama (see Theorem \ref{th:Uchiyama} below) in the intervals $\mathcal{I}_j^{**}$, $j\ge1$. For the converse, we start by decomposing a function in $H^1_{\mathcal{L}}$ by using a partition of unity with functions supported in the intervals $\mathcal{I}_j^*$. Then, again suitable kernel estimates are applied, together with the fact that the operators localized in $\mathcal{I}_j$, for $j\ge1$, fall into Uchiyama's theory. The subtle point is the analysis on the interval $\mathcal{I}_0$. Our idea to tackle this case is to compare the Poisson semigroup $\mathcal{P}_t$ with the Poisson semigroup for the Bessel operator on $(0,\infty)$ by using Duhamel's principle. For the case of the operator $L$ the ideas are analogous but simpler, because Uchiyama's result works well in all the intervals $\mathcal{J}_j$.

The paper is organized as follows. The definitions of the orthogonal systems that we study are in Section \ref{sec3}, as well as the necessary estimates for the related Poisson kernels. In Section \ref{sec2} we briefly recall the notion of local Hardy spaces on spaces of homogeneous type and we state the theorem by Uchiyama, which is one of the main tools in the proofs of Theorems \ref{thmA} and \ref{thmB}.  In Section \ref{sec:near0} we compare the Poisson semigroups in the discrete and continuous Bessel settings by using Duhamel's formula and the results of \cite{Betancor-Dziubanski-Torrea}. Finally, in Sections \ref{sec4} and \ref{sec5} we prove Theorems \ref{thmA} and \ref{thmB}, respectively.

\section{Auxiliary estimates}\label{sec3}

\subsection{The orthogonal expansions}\label{ortt}

For $\nu>-1$, let $\{\lambda_{n,\nu}\}_{n\ge1}$ denote the sequence of successive positive zeros of the Bessel function $J_{\nu}$ and consider
    $$\phi_n^{\nu} (x) = d_{n,\nu} \lambda_{n,\nu}^{1/2} J_\nu (\lambda_{n,\nu}  x)
    x^{-\nu},\qquad\psi_n^{\nu} (x) = x^{\nu+1/2} \phi_n^{\nu}(x),$$
where $x\in (0,1)$ and $d_{n,\nu}=\sqrt2{|\lambda_{n,\nu}J_{\nu+1}(\lambda_{n,\nu})|}^{-1}$. It is well known that the systems $\{\phi_n^{\nu}\}_{n=1}^\8$ and $\{\psi_n^{\nu}\}_{n=1}^\8$ form complete orthonormal bases of $L^2((0,1),\mu)$ and $L^2(0,1)$, respectively.

The functions $\phi_n^{\nu}$ and $\psi_n^{\nu}$ are eigenfunctions of the differential operators \eqref{eq:LRedonda} and \eqref{eq:LCuadrado}, namely,

    \begin{equation*}
    \mathcal{L}\phi_n^{\nu}(x)=\lambda_{n,\nu}^2\phi_n^{\nu}(x), \qquad  L \psi_n^{\nu} (x)=\lambda_{n,\nu}^2\psi_n^{\nu} (x).
    \end{equation*}

The operators $\mc L$ and $ L$ posses natural self-adjoint extensions, that for convenience we still denote by $\mc L$ and $L$, with domains
    \begin{align*}
    &\operatorname{Dom}(\mc L) = \Big\{f\in L^2 ((0,1), \mu) \ : \ \sum_{n=1}^\8 \lambda_{n,\nu}^4 |\langle f, \phi_n^\nu \rangle_{\mu}|^2 <\8\Big\},\\
    &\operatorname{Dom}(L) = \Big\{f\in L^2 (0,1) \ : \ \sum_{n=1}^\8 \lambda_{n,\nu}^4 |\langle f, \psi_n^\nu \rangle|^2 <\8\Big\}.
    \end{align*}
Obviously, $\overline{\operatorname{Dom}(\mc L)} = L^2 ((0,1), \mu)$ and $\overline{\operatorname{Dom}(L)}=L^2 (0,1)$. From now on by using $\mc L$ or $L$ we always mean the self-adjoint extensions. In Lemmas \ref{domain1}, \ref{domain2} and \ref{lee3} we provide arguments that for some functions $f$ the images $\Ll f$ and $Lf$ can be obtained by simply applying the differential formulas \eqref{eq:LRedonda} and \eqref{eq:LCuadrado}.

For $t>0$ and $x\in(0,1)$, the Poisson integrals related to ${\Ll}$ and $L$ are defined by
    \begin{align}\label{Poissone_formula}
    &\mathcal{P}_tf(x)= e^{-t\sqrt{\mathcal{L}}}f(x)=\int_0^1\mathcal{P}_t(x,y)f(y)\,d\mu(y),
    &P_t g(x)= e^{-t\sqrt{L}}g(x)=\int_0^1P_t(x,y)g(y)\,dy,
    \end{align}
where $f\in L^2((0,1),\mu)$, $g\in L^2(0,1)$, and the corresponding Poisson kernels are
    \begin{equation}\label{Poisson kernel}
    \mathcal{P}_t(x,y) = \sum_{n=1}^{\infty} e^{-t\lambda_{n, \nu}} \phi_{n}^{\nu}(x)\phi_{n}^{\nu}(y), \qquad P_t(x,y)= (xy)^{\nu+1/2} \mathcal{P}_t(x,y).
    \end{equation}

 We shall use some basic facts about Bessel functions that we collect here. We refer the reader to \cite{Lebedev, Watson} for details. The Bessel function $J_\nu$ satisfies:

\begin{align}
\label{eq:deriv}
&\frac{d}{dx}\(x^{-\nu}J_\nu(x)\)= - x^{-\nu} J_{\nu+1}(x),&&\hbox{for}~x\in\mathbb{R},\\
\label{eq:asymp_zero}
&J_\nu(x)=O(x^\nu), &&\hbox{for}~|x|\to 0,\\
\label{eq:asymp_inf}
&J_\nu(x)=\sqrt{2}(\pi x)^{-1/2}(\cos(x+D_\nu)+O(x^{-1})), &&\hbox{for}~|x|\to\infty,
\end{align}
where $D_\nu=-(\nu\pi/2+\pi/4)$. Moreover,
\begin{equation} \label{eq:zero-cte}
\lambda_{n,\nu}=O(n), \qquad d_{n,\nu}=O(1).
\end{equation}

Let us notice that the operator $L$ can be decomposed as $L = \delta_L^* \delta_L$, where
    \begin{equation}\label{deltaL}
    \delta_L = -\frac{d}{dx}+\frac{\nu+1/2}{x},
    \end{equation}
and $\delta^*_L=\frac{d}{dx}+\frac{\nu+1/2}{x}$ is the formal adjoint of $\delta_L$ in $L^2(0,1)$.

\subsection{Estimates for the Poisson kernel $\Pp_t(x,y)$}

The following estimates were proved in \cite{Nowak-Roncal}.
    \begin{lem}\label{Lem:SharpNatural}
    For $t\leq 1$ we have
        \begin{equation*}
            \mathcal{P}_t(x,y) \simeq \(\frac{1}{t^2+x^2+y^2}\)^{\nu+1/2}\(\frac{(1-x)(1-y)}{t^2+(1-x)^2+(1-y)^2}\)\frac{t}{t^2+|x-y|^2},
        \end{equation*}
    and for $t>1$ we have
    $$\mathcal{P}_t(x,y) \simeq (1-x)(1-y) e^{-t\lambda_{1,\nu}}.$$
    \end{lem}

We shall also need the estimate on $\frac{\partial}{\partial x}\mathcal{P}_t(x,y)$ contained in Lemma \ref{PropositionSixNonLeb}. The result readily follows from Lemma \ref{Lem:PropositionSix} below via the relation $(xy)^{\nu+1/2} \frac{\partial}{\partial x}\mathcal{P}_t(x,y) = \delta_L P_t(x,y)$, see \eqref{Poisson kernel}, \eqref{eq:deriv}, and \eqref{deltaL}.

    \begin{lem}\label{PropositionSixNonLeb}
    There exists a constant $C$ such that for $t>0,~x,y\in(0,1)$ we have
    $$\Big|\frac{\partial}{\partial x}\mathcal{P}_t(x,y)\Big| \leq C  \frac{(xy)^{-\nu-1/2}}{t^2+|x-y|^{2}}.$$
    \end{lem}

\subsection{Estimates for the Poisson kernel $P_t(x,y)$}

Recall that $P_t(x,y)$ and $\Pp_t(x,y)$ are related by \eqref{Poisson kernel}. Thus, the following lemma is an immediate consequence of Lemma \ref{Lem:SharpNatural} (see also {\cite{Nowak-Roncal}}).

    \begin{lem}
    \label{Lem:SharpLebesgue}
    For $t\le1$ we have
    \begin{equation*}
        P_t(x,y) \simeq \Big(\frac{xy}{t^2+x^2+y^2}\Big)^{\nu+1/2}\Big(\frac{(1-x)(1-y)}{t^2+(1-x)^2+(1-y)^2}\Big)\frac{t}{t^2+|x-y|^2},
    \end{equation*}
    and for $t>1$ we have
        $$
        P_t(x,y) \simeq (xy)^{\nu+1/2} (1-x)(1-y) e^{-t\lambda_{1,\nu}}.
        $$
    \end{lem}

Now we prove the estimate for $\delta_LP_t(x,y)$.

    \begin{lem}\label{Lem:PropositionSix}
    There exists a constant $C$ such that
        \begin{equation*}
         \, |\delta_LP_t(x,y)|\leq C \frac{1}{t^2+|x-y|^{2}}.
        \end{equation*}
    \end{lem}

    \begin{proof}
    From \eqref{Poisson kernel} and \eqref{eq:deriv} we have
        \begin{equation}\label{delta Pt}
        \delta_L P_t(x,y)=\sum_{n=1}^\infty e^{-t\lambda_{n,\nu}}d_{n,\nu}^2\lambda_{n,\nu}^2(xy)^{1/2}J_{\nu+1}(\lambda_{n,\nu}x)J_{\nu}(\lambda_{n,\nu}y).
        \end{equation}
    For $|x-y|>t$ the conclusion follows by taking $a=2$, $b=\nu+1$ and $c=\nu$ in \cite[Proposition 6]{Ciaurri-Roncal}.

    Assume now that $|x-y|\leq t$. Let $N_1\leq N_2$ be positive integers that will be fixed later on. We split the sum in \eqref{delta Pt} into three parts
        \begin{equation*}
        \delta_LP_t(x,y)=\sum_{n=1}^{N_1}~+\sum_{n=N_1+1}^{N_2}~+\sum_{N_2+1}^\infty~=:\Sigma_0+\Sigma_1+\Sigma_2.
        \end{equation*}
    We finish the proof by considering three cases.

    \noindent\textbf{Case 1.} Assume that $0<y<x/2$. Then, $t>|x-y|\simeq x$. Take $N_1=\lfloor 1/x\rfloor$ and $N_2 = \lfloor 1/y \rfloor$. By using \eqref{eq:asymp_zero}, \eqref{eq:asymp_inf}, and \eqref{eq:zero-cte},
    \begin{align*}
        |\Sigma_0| &\leq C(xy)^{1/2}\sum_{n=1}^{N_1}e^{-ctn}n^2(nx)^{\nu+1}(ny)^\nu\leq C\frac{x^{2\nu+2}}{t^{2\nu+3}}\sum_{n=1}^{N_1}e^{-ctn}(tn)^{2\nu+3} \\
         &\leq \frac{C}{t}\sum_{n=1}^\infty e^{-c'tn} = \frac{Ce^{-c't}}{t(1-e^{-c't})}\leq \frac{C}{t^2},\\
        |\Sigma_1| &\leq C(xy)^{1/2}\sum_{n=N_1+1}^{N_2}e^{-ctn}n^2(nx)^{-1/2}(ny)^\nu\leq C\frac{y^{\nu+1/2}}{t^{\nu+3/2}}\sum_{n=N_1+1}^{N_2}e^{-ctn}(tn)^{\nu+3/2} \\
         &\leq \frac{C}{t}\sum_{n=1}^\infty e^{-c'tn} = \frac{Ce^{-c't}}{t(1-e^{-c't})}\leq \frac{C}{t^2},\\
    |\Sigma_2|&\leq C(xy)^{1/2}\sum_{n=N_2+1}^\infty e^{-ctn}n^2(nx)^{-1/2}(ny)^{-1/2}\leq \frac{C}{t}\sum_{n=N_2+1}^{\infty}e^{-ctn}(tn)\leq\frac{C}{t^2}.
    \end{align*}

    \noindent\textbf{Case 2.} Assume that $x/2\leq y\leq2x$. Here $|x-y|<x \simeq y$. We take $N_1=N_2=\lfloor 1/x\rfloor$. We will consider first the subcase when $t>x$. From \eqref{eq:asymp_zero}, \eqref{eq:asymp_inf}, and \eqref{eq:zero-cte},
    \begin{align*}
    |\Sigma_0|&\leq C(xy)^{1/2}\sum_{n=1}^{N_1}e^{-ctn}n^2(nx)^{\nu+1}(ny)^\nu\leq C\frac{x^{2\nu+2}}{t^{2\nu+3}}\sum_{n=1}^{N_1}e^{-ctn}(tn)^{2\nu+3}\leq \frac{C}{t^2},\\
    |\Sigma_2|&\leq C(xy)^{1/2}\sum_{n=N_1+1}^\infty e^{-ctn}n^2(nx)^{-1/2}(ny)^{-1/2}=\frac{C}{t}\sum_{n=N_1+1}^\infty e^{-ctn}(tn)\leq\frac{C}{t^2}.
    \end{align*}
    Consider now the second subcase, that is, when $|x-y|<t\leq x$. Again from \eqref{eq:asymp_zero} and \eqref{eq:zero-cte},
    \begin{align*}
        |\Sigma_0| &\leq C(xy)^{1/2}\sum_{n=1}^{N_1}e^{-ctn}n^2(nx)^{\nu+1}(ny)^\nu\leq Cx^{2\nu+2}\sum_{n=1}^{N_1}n^{2\nu+3}\\
         &\leq Cx^{2\nu+2}N_1^{2\nu+4}\leq Cx^{-2}\leq Ct^{-2},
    \end{align*}
    and the estimate of $\Sigma_2$ is obtained exactly as in the first subcase.

    \noindent
    \textbf{Case 3.} Assume that $0<x<y/2$. Under this assumption, $t>|x-y|\simeq y$. $N_1=\lfloor 1/y\rfloor$ and $N_2 = \lfloor 1/x \rfloor$. Now, the estimates follow by the same arguments as in Case 1.
    \end{proof}

The following corollary is a direct consequence of the symmetry of $P_t(x,y)$, Lemma \ref{Lem:PropositionSix} and \eqref{deltaL}.
    \begin{cor}\label{corrrr}
    There exists a constant $C$ such that
    \begin{equation*}
        \left|\frac{\partial}{\partial y}P_t(x,y)\right|\leq C\frac{1}{t^2+|x-y|^{2}}+\frac{C}{y}|P_t(x,y)|.
        \end{equation*}
    \end{cor}

\section{Local Hardy spaces on spaces of homogeneous type}\label{sec2}

In this section we present some facts about local Hardy spaces on spaces of homogeneous type that will be used in the proofs of our main results. The reader may find more details and references in \cite{Coifman-Weis,Macias-Segovia,Uchiyama}.

Let $(X,d,\sigma)$ be a space of homogeneous type. Additionally, suppose that there exists a constant $A>1$ satisfying
    $$A^{-1}r\le \sigma(B_d(x,r))\le A r,
    $$
for $x\in X$ and $0<r<\sigma(X)$, where $B_d(x,r)=\{y\in X:d(x,y)<r\}$. Moreover, assume that there exists a continuous function $K(r,x,y)$, for  $(t,x,y)\in(0,\sigma(X))\times X\times X$, and positive constants $\gamma_1, \gamma_2, \gamma_3$, such that
    \begin{align}
        \label{UchiRep1} &K(r,x,x)\geq \frac{1}{Ar}>0,\\
        \label{UchiRep2} &0\leq K(r,x,y)\leq \frac{A}{r}\left(1+\frac{d(x,y)}{r}\right)^{-1-\gamma_1},
    \end{align}
and, whenever $d(y,z)<(r+d(x,y))/(4A)$, we have
    \begin{equation}\label{UchiRep3}
          |K(r,x,y)-K(r,x,z)|\leq \frac{A}{r}\left(\frac{d(y,z)}{r}\right)^{\gamma_2}\left(1+\frac{d(x,y)}{r}\right)^{-1-\gamma_3}.
    \end{equation}

Consider (see \cite{Uchiyama}) the maximal function
    \begin{equation}\label{eq:maximalUchi}
        f^{(+)}(x)=\sup_{0<r<\sigma(X)}|K_{r}f(x)|=\sup_{0<r <\sigma(X)}\Big|\int_{X}K(r,x,y)f(y)d\sigma(y)\Big|.
    \end{equation}

Let $H^1(X):=\{f:\|f\|_{H^1(X)}:=\|f^{(+)}\|_{L^1(X)}<\infty\}$. Assume now that $\sigma(X)<\infty$. We say that a function $a$ is an atom for $H^1_{\mathrm{at}}(X)$ if either
    \begin{enumerate}
               \item[(i)] there exists a ball $B_d(x_0,R)\subset X$ such that
        $$\supp \, a \subseteq B_d(x,R), \qquad \|a\|_{\8} \leq \sigma(B_d(x,R))^{-1},\quad \text{ and } \quad \int_X a(x)\, d\sigma(x) =0,$$
        or
         \item[(ii)] $a(x)=\sigma(X)^{-1}\chi_X(x)$.
    \end{enumerate}
After defining atoms, the space $H^1_{\mathrm{at}}(X)$ and its norm are defined in the usual way.

\begin{thm}[{\cite[Corollary 1']{Uchiyama}}]
        \label{th:Uchiyama}
        Let $(X,d,\sigma)$ be a space of homogeneous type equipped with a kernel $K(r,x,y)$, $0<r<\sigma(X)$, satisfying \eqref{UchiRep1}--\eqref{UchiRep3}. Then there exists constants $C_1, C_2>0$ depending only on $A$ and $\gamma_1, \gamma_2, \gamma_3$ such that
                $$C_1\|f\|_{H^1(X)}\le\|f\|_{H^1_{\mathrm{at}}(X)}\le C_2\|f\|_{H^1(X)}.$$
    \end{thm}

\section{Analysis of $\mathcal{M}$ for $x,y$ near $0$}\label{sec:near0}

In this section we analyze the maximal function $\sup_{0<t<1}|\mathcal{P}_t f(x)|$ for $x\in \mathcal{I}_0^{**}$, where the function $f$ is such that $\mathrm{supp} \,f \subseteq \mathcal{I}_0^{**}$, see Subsection \ref{subsection L redondo} for the definition of $\mathcal{I}_0^{**}$. We are going to use Duhamel's formula to compare $\Pp_t$ with the Poisson semigroup $\PPp_t$ of the Bessel operator $\mathbf L$ acting on $(0,\8)$. In order to give the precise definitions, consider the Bessel differential operator
    \begin{equation}\label{bbess}
   \mathbf{L}=-\frac{d^2}{dx^2}-\frac{2\nu+1}{x}\frac{d}{dx}.
    \end{equation}
The operator $\mathbf L$ has a self-adjoint, densely defined extension on $L^2((0,\infty),\mu)$ that for convenience we still denote by $\mathbf L$. This extension is given by $\Hh(\mathbf L f)(\xi) = \xi^2 \Hh f (\xi)$, with
    \begin{align*}
    \operatorname{Dom} (\mathbf L) = &\left\{ f\in L^2((0,\infty),\mu): \xi^2 \mathcal H f(\xi)\in L^2((0,\8),\mu)\right\}.
    \end{align*}
Here $\mathcal{H}$ denotes the Hankel transform, which is the isometry on $L^2((0,\8), \mu)$ defined by
    $$ \mathcal H f(\xi)= c_\nu \int_0^\infty \varphi^{\nu }(\xi y)f(y)\, d\mu (y),$$
and $\varphi^{\nu }(\xi y)= (\xi y)^{-\nu} J_\nu (\xi y)$, for $\xi>0$.
The Poisson semigroup $\PPp_t$ related to $\mathbf L$ can be expressed in the following way
    $$\mathcal H (\PPp_t f)(\xi)= \mathcal{H}(e^{-t\sqrt{\mathbf L}}f)(\xi)=e^{-t\xi} \mathcal H f(\xi),\quad\xi>0.$$
The Hardy spaces related to the Bessel operator were investigated in \cite{Betancor-Dziubanski-Torrea} and \cite{Preisner}. The main result of this section is the following.

    \begin{thm}\label{theorem}
    There exists a constant $C$ such that, for any $f\in L^1(\mathcal{I}_0^{**}, \mu)$,
    $$\Big\| \sup_{0<t<1} |\PPp_t f(x)-\mathcal P_t f(x)|\Big\|_{L^1(\mathcal{I}_0^{**}, \mu)} \leq C \| f\|_{L^1(\mathcal{I}_0^{**}, \mu)}.$$
    \end{thm}

To prove Theorem \ref{theorem} we shall consider the heat semigroups generated by ${\Ll}$ on $L^2((0,1),\mu)$ and by $\mathbf{L}$ on $L^2((0,\8),\mu)$. We have that
    \begin{align*}
    &\mathcal T_tf(x)=e^{-t\Ll}f(x)=\int_0^1 \mathcal T_t(x,y)f(y)\, d\mu (y),\\
    &\mathcal T_t(x,y) = \sum_{n=1}^\infty e^{-t\lambda_{n,\nu}^2} \phi_n^\nu (x)
    \phi_n^\nu (y),\quad x,y\in(0,1),~t>0,
    \end{align*}
and
    \begin{align*}
    &\mathbf T_tf(x)=e^{-t\mathbf L}f(x)=\int_0^\infty \mathbf T_t(x,y) f(y)\, d\mu (y),\\
    &\mathbf T_t(x,y) = (2t)^{-1} \exp\left(-\frac{x^2+y^2}{4t}\right) I_{\nu} \left(\frac{xy}{2t}\right) (xy)^{-\nu},\quad x,y\in(0,\infty),~t>0.
    \end{align*}
It is well known that
$ \mathcal H (\mathbf T_t f)(\xi)=e^{-t\xi^2} \mathcal H f(\xi)$ on $L^2((0,\infty),\mu)$.

In Lemmas \ref{domain1}, \ref{domain2}, \ref{lee3} we study the domains of the operators $\mc L$ and $\mathbf L$.
    \begin{lem} \label{domain1}
    Let $f\in C^2_c[0,\infty)$ such that $f'(0^+)=0$. Then $f\in \operatorname{Dom} (\mathbf L)$ and
  $$
    \mathbf L f(x)=-f''(x) - \frac{2\nu +1}{x}f'(x).
    $$
    \end{lem}

    \begin{proof}
    Let $g\in C^2[0,\infty)$, $g'(0^{+})=0$. By integrating by parts,
  \begin{equation}\label{3.1}
    \begin{aligned}
    \int_0^{\infty}\Big(-f''(x) - \frac{2\nu +1}{x}f'(x)\Big)g(x)\,d\mu(x)&=\int_0^\infty f'(x)g'(x) \, d\mu (x)\\
    &=\int_0^\infty f(x)\Big(-g''(x) - \frac{2\nu +1}{x}g'(x)\Big)\,d\mu(x).
    \end{aligned}
    \end{equation}
For every $\xi>0$ the function $\varphi^{\nu }_\xi (x) = \varphi^{\nu }(\xi x)$ is $C^\infty[0,\8)$,  $\frac{d}{dx} \varphi^{\nu}_{\xi} (x)=0$ for $x=0^{+}$, and verifies $\frac{d^2}{dx^2} \varphi^\nu_\xi ( x)+\frac{2\nu+1}{x}\frac{d}{dx}\varphi^\nu_\xi (x)=-\xi^2 \varphi^\nu_\xi(x)$. From \eqref{3.1},
\begin{equation*}
\begin{split}
-\xi^2 \mathcal H f(\xi) & = \int_0^\infty f(x) \Big(\frac{d^2}{dx^2} \varphi^\nu_\xi  (x)+\frac{2\nu+1}{x}\frac{d}{dx}\varphi^\nu_\xi ( x)\Big) d\mu(x)\\
&=\int_0^\infty \Big(f''(x) + \frac{2\nu +1}{x}f'(x)\Big)\varphi^\nu_\xi( x)d\mu (x)\\
&= \mathcal H\Big(f'' + \frac{2\nu +1}{x}f'\Big)(\xi).
\end{split}
\end{equation*}
Note that $f''(x) + \frac{2\nu +1}{x}f'(x)$ is continuous on $[0,\infty)$ with compact support. This and the identity above give that $f\in \operatorname{Dom}(\mathbf L)$. Since, by definition, $\mathcal H(\mathbf L f)(\xi)= \xi^2 \mathcal H f(\xi)$ we obtain the conclusion.
    \end{proof}

    \begin{lem} \label{domain2}
    Let $f\in L^2((0,1),\mu)$. Then $\mathcal T_t f\in \operatorname{Dom}(\mathcal L)\cap C^2[0,1]$ and
      $$\mathcal L \mathcal T_t f(x)=-\Big(\frac{d^2}{d x^2} +\frac{2\nu+1}{x}\frac{d}{dx} \Big) (\mathcal T_t f)(x).$$

    \end{lem}

    \begin{proof}
    For $f\in L^2(0,1)$ we write $f= \sum_{n=1}^\8 \langle f, \phi_n^\nu \rangle_\mu \phi_n^\nu$, so $\|f\|^2_{L^2(\mu)} = \sum_{n=1}^\8 |\langle f, \phi_n^\nu \rangle_\mu|^2 <\8$. Then, by definition, $\mc T_t f(x) = \sum_{n=1}^\8 e^{-t\lambda_{n,\nu}^2} \langle f, \phi_n^\nu \rangle_\mu \phi_n^\nu(x)$. One easily verifies that $\mc T_t f(x) \in C^2[0,1]$,
    \begin{align*}
    -\Big(\frac{d^2}{d x^2} +\frac{2\nu+1}{x}\frac{d}{dx} \Big) (\mathcal T_t f)(x)&= \sum_{n=1}^\8 e^{-t\lambda_{n,\nu}^2} \langle f, \phi_n^\nu \rangle \Big(-\frac{d^2}{d x^2} -\frac{2\nu+1}{x}\frac{d}{dx} \Big) \phi_n^\nu(x)\\
    & = \sum_{n=1}^\8 \lambda_{n,\nu}^2 e^{-t\lambda_{n,\nu}^2} \langle f, \phi_n^\nu \rangle \phi_n^\nu(x) = \mc L \mc T_t f (x),
    \end{align*}
    and, clearly, $\mathcal T_t f\in \operatorname{Dom}(\mathcal L)$.
    \end{proof}

Let us set $\widetilde{\phi}_n^\nu (x)=\phi_n^{\nu}(x)\chi_{(0,1)}(x)$, for $x>0$. On $L^2((0,\infty),\mu)$ consider the operator
    $$\widetilde{\mathcal T}_tf(x)=\int_0^\infty  \widetilde{\mathcal T}_t(x,y)f(y)\, d\mu (y),$$
with the integral kernel
    \begin{equation}\label{defTtilda}
    \widetilde{\mathcal{T}}_t(x,y)=\sum_{n=1}^\infty e^{-t\lambda_{n,\nu}^2} \widetilde\phi_n^\nu (x)
    \widetilde\phi_n^\nu (y), \quad  t,x,y\in(0,\8).
    \end{equation}
Clearly, $\widetilde{\mathcal T}_t$ is not a $C_0-$semigroup on $L^2((0,\infty),\mu)$ because $\lim_{t\to 0^+}\wt{ \mc T_t}f$  is not the identity operator in $L^2((0,\infty),\mu)$. Nevertheless, note that $\widetilde{\mathcal T}_tf(x) = \mathcal{T}_tf(x)$ when $x\in(0,1)$, $f\in L^2((0,\infty),\mu)$ and $\operatorname{supp}f\subseteq[0,1]$.

Let $M_h$ denote the multiplication operator, $M_h f(x)=h(x) f(x)$. Fix a function $\rho\in C^\infty[0,\8)$ such that $\rho(x)=1$ for $x\in\mc I_0^{**}$, $\rho (x)= 0$ for $x\not \in \mc I_0^{***}$.

    \begin{lem} \label{lee3} For every $s>0$ and $f\in L^2((0,\infty),\mu)$ the function $M_\rho\widetilde{\mathcal{T}}_s f$ is of class $C^2[0,1]$  and  $(M_\rho\widetilde{\mathcal T}_s f)'(0^+) =0$.
    \end{lem}

    \begin{proof}
    From the definition in \eqref{defTtilda},
        \begin{equation*}
        M_\rho\widetilde{\mathcal{T}}_s f(x) = \rho(x) \sum_{n=1}^\8 e^{-s\lambda_{n,\nu}^2} \phi_n^\nu(x) \int_0^1 \phi_n^\nu(y) f(y) \, d\mu(y).
        \end{equation*}
    Observe that $|\int_0^1 \phi_n^\nu \, f \, d\mu| \leq C$ and that $(\phi_n^\nu)'(0^+)=0$. The lemma follows from the properties of $\rho$ and $\phi_n^\nu$.
    \end{proof}

We are ready to consider the family of operators $\mathbf T_{t-s} M_\rho \widetilde{\mathcal T}_s $, $0<s<t$, acting on $L^2((0,\infty),\mu)$-functions with support contained in $\mc I_0^{**}$. A direct calculation leads to
    \begin{equation*}
    \begin{split}
    \frac{d}{ds}\Big(\mathbf T_{t-s} M_\rho\widetilde{ \mathcal T}_s f\Big) &=
    \mathbf T_{t-s} \( \mathbf L( M_\rho \widetilde{\mathcal T}_s f) - M_\rho \mathcal L \widetilde{\mathcal T}_s f\)\\
    &= -\mathbf T_{t-s} \( M_{\rho''} \widetilde{\mathcal T}_sf+2 M_{\rho '}  (\widetilde{\mathcal T}_s f)'+ M_{\frac{2\nu +1}{x} \rho'} \widetilde{\mathcal T}_s f\).
    \end{split}
    \end{equation*}
The last equality is justified by Lemmas \ref{domain1}, \ref{domain2}, and \ref{lee3}.
Consequently, by integrating the equation above we get Duhamel's identity
    \begin{equation} \label{eee1}
    \begin{split}
     M_\rho \widetilde{\mathcal T}_t f - \mathbf T_t M_\rho f
    = \int_0^t
    -\mathbf T_{t-s} M_{\rho''} \widetilde{\mathcal T}_sf\, ds -2\int_0^t\mathbf T_{t-s} M_{\rho '}  (\widetilde{\mathcal T}_s f)'\, ds -\int_0^t\mathbf T_{t-s} M_{\frac{2\nu +1}{x} \rho'}\widetilde{ \mathcal T}_s f\, ds.
    \end{split}
    \end{equation}
Integrating by parts we obtain
    \begin{equation}\label{eee2}
    \begin{split}
    -\int_0^t\mathbf T_{t-s} M_{\rho '}  (\widetilde{\mathcal T}_s f)'(x)\, ds
    &=- \int_0^t \int_0^\infty  \mathbf T_{t-s}(x,z)\rho'(z) z^{2\nu+1} \frac{d}{dz}(\widetilde{\mathcal T}_s f)(z)\, dz\, ds\\
    &= \int_0^t \int_0^\infty \Big( \frac{d}{dz} \mathbf T_{t-s} (x,z)\Big) \rho'(z)z^{2\nu +1} \widetilde{\mathcal T}_s f(z)\, dz\, ds \\
    &\quad  +\int_0^t \int_0^\infty  \mathbf T_{t-s} (x,z)\rho''(z)z^{2\nu +1}\widetilde{\mathcal T}_sf(z) \, dz\, ds\\
    & \quad +\int_0^t \int_0^\infty  \mathbf T_{t-s} (x,z)\rho'(z)\frac{2\nu +1}{z} z^{2\nu +1}\widetilde{\mathcal T}_sf(z) \, dz\, ds.
    \end{split}
    \end{equation}
Combining \eqref{eee1} and \eqref{eee2} together,
    \begin{equation}\label{formula1}
    \begin{split}
    M_\rho \widetilde{\mathcal T}_t f(x) - \mathbf T_t M_\rho f(x)
    & =  \int_0^t \int_0^\infty  \mathbf T_{t-s} (x,z)\rho''(z)\widetilde{\mathcal T}_sf(z) \, d\mu (z)\, ds\\
     & \quad + 2\int_0^t \int_0^\infty \Big( \frac{d}{dz} \mathbf T_{t-s} (x,z)\Big) \rho'(z)  \widetilde{\mathcal T}_s f(z)\, d\mu (z)\, ds \\
    & \quad + \int_0^t \int_0^\infty  \mathbf T_{t-s} (x,z)\rho'(z)\frac{2\nu +1}{z} \widetilde{\mathcal T}_sf(z) \, d\mu (z) \, ds\\
    &= R_t^{[1]} f(x)+R_t^{[2]}f(x)+R_t^{[3]}f(x).
    \end{split}
    \end{equation}

We shall prove that the operators $\sup_{0<t<1} \chi_{\mc I^{**}}|R^{[j]}_t f|$, $j=1,2,3,$ are well-defined and bounded on $L^1(\mc I_0^{**}, \mu)$. It is an exercise to check that the measure $\mu$ on $(0,\8)$ satisfies
    \begin{equation}\label{measure}
    \mu(B(x,\sqrt{t}))\simeq
    \left\{
        \begin{array}{ll}
          x^{2\nu+1}\sqrt{t}, & \sqrt{t}\leq 2x, \\
          (\sqrt{t})^{2\nu+2}, & \sqrt{t}>2x.
        \end{array}
      \right.
    \end{equation}
We will also need well-known Gaussian estimates on $\mathbf T_t$ and $\mc T_t$.

    \begin{lem}\label{Lem:Gaussian}
    There exist constants $C,c>0$ such that the integral kernels $\mathbf T_t (x,y)$ and $\mathcal T_t(x,y)$ satisfy:
        \begin{itemize}
        \item[(a)]$\displaystyle 0<\mathbf T_t(x,y)\leq \frac{C}{\mu (B(x,\sqrt{t}))}\, e^{-c|x-y|^2\slash t}$, \qquad \qquad for $x,y,t\in(0,\infty)$,
        \item[(b)]$\displaystyle  \Big|\frac{\partial}{\partial y}\mathbf T_t(x,y)\Big|\leq \frac{C}{\sqrt{t}\mu (B(x,\sqrt{t}))}\, e^{-c|x-y|^2\slash t}$, \qquad \ \ for $x,y,t\in(0,\infty)$,
        \item[(c)]$\displaystyle  0\leq \mathcal{T}_t(x,y)\leq \frac{C}{\sqrt{t}(t\vee xy)^{\nu+1/2}}e^{-c(x-y)^2/t}$, \quad \ \ \ \ \ for $x,y,t\in(0,1)$,
        \item[(d)]$\displaystyle \mathcal{T}_t(x,y)\simeq (1-x)(1-y)e^{-t\lambda_{1,\nu}^2}$, \quad \qquad \qquad \qquad for $x,y\in(0,1)$, $t\geq1$.
        \end{itemize}
    \end{lem}

    \begin{proof}
    For (a) and (b) see \cite[Lemma~4.3]{DPW}, while for (c) and (d) see \cite{Nowak-Roncal-Sharp}.
    \end{proof}

Observe that, from \eqref{defTtilda}, we can express $R_t^{[j]}$, $j=1,2,3$, as integral operators.
    \begin{lem}\label{lem:Rj's}
    Let $R_t^{[j]}(x,y)$ denote the integral kernel of $R_t^{[j]}$, $j=1,2,3$. Then there is a constant $C$ such that for all $x,y\in \mc I_0^{**}, 0<t<1$,
    $$ |R_t^{[j]}(x,y)|\leq C.$$
    \end{lem}

    \begin{proof}
    Note that  $ \operatorname{supp}\, \rho'' \subset (\frac{1}{2}, \frac{2}{3})$. From Lemma \ref{Lem:Gaussian} and \eqref{measure},
        \begin{equation*}
        \begin{split}
        | R_t^{[1]}(x,y)|& \leq C\int_0^t \int_0^\infty
        \frac{e^{-c|x-z|^2\slash (t-s)}}{\mu (B(x,\sqrt{t-s}))}|\rho''(z)|
        \frac{e^{-c|y-z|^2\slash s}}{(s\vee yz)^{\nu+1/2}\sqrt{s}}\, d\mu (z)\, ds\\
        &\leq C\int_0^t \int_{1\slash 2}^{2\slash {3}}
        \frac{e^{-c'\slash (t-s)}}{\mu (B(x,\sqrt{t-s}))}
        \frac{e^{-c'\slash s}}{(s\vee yz)^{\nu+1/2}\sqrt{s}}\, d\mu (z)\, ds\\
        &\leq C\int_0^t
        \frac{e^{-c'\slash (t-s)}}{(\sqrt{t-s})^{2\nu+2}}
        \frac{e^{-c'\slash s}}{(\sqrt{s})^{2\nu+2}}\, ds\leq C.
        \end{split}
        \end{equation*}
    The proofs for $| R_t^{[2]}(x,y)|$ and  $| R_t^{[3]}(x,y)|$ are done in a similar way.
    \end{proof}

Lemma \ref{lem:Rj's} implies that the operators $\sup_{0<t<1} \chi_{\mc I_0^{**}}|R^{[j]}_t f|$, $j=1,2,3$, are bounded on $L^1(\mc I_0^{**}, \mu)$. As a consequence of identity \eqref{formula1} we arrive at the following result.

    \begin{cor}\label{heat}
    There is a constant $C$ such that for $f\in L^1((0,\infty),\mu)$ with $\operatorname{supp}f\subset\mc I_0^{**}$ we have
    $$ \Big\|\sup_{0<t<1} |\mathcal T_t f -\mathbf T_t f|\Big\|_{L^1(\mc I_0^{**},\mu)}\leq C\| f\|_{L^1(\mc I_0^{**},\mu)}.$$
    \end{cor}

    \begin{proof}[Proof of Theorem \ref{theorem}]
    We shall use the principle of subordination to pass the estimates from the heat kernels to the Poisson kernels. For $0<t<1$, $x\in \mc I_0^{**}$ and $f\in L^1((0,\infty),\mu)$ with $\operatorname{supp}f\subset \mc I_0^{**}$,
        \begin{equation*}
        \begin{split}
        |\mathcal P_tf(x)-\mathbf P_tf(x)| &= C \Big| \int_0^\infty \frac{e^{-u}}{\sqrt u}
        (\mathcal T_{t^2\slash(4u)} -\mathbf T_{t^2\slash (4u)} ) f(x)\, du\Big|\\
        &\leq  C \Big| \int_{t^2\slash 4}^\infty \frac{e^{-u}}{\sqrt u}
        (\mathcal T_{t^2\slash (4u)} -\mathbf T_{t^2\slash (4u)} ) f(x)\, du\Big|\\
        & \quad
        + C \Big| \int_0^{t^2\slash 4}\frac{e^{-u}}{\sqrt u}
        (\mathcal T_{t^2\slash (4u)} -\mathbf T_{t^2\slash (4u)} ) f(x)\, du\Big|
        \\
        &\leq C  \int_{t^2\slash 4}^\infty \frac{e^{-u}}{\sqrt u}
        \sup_{0<s<1} |(\mathcal T_{s} -\mathbf T_{s} ) f(x)|\, du \\
        & \ \ +C \Big| \int_0^{t^2\slash 4}\frac{e^{-u}}{\sqrt u}
        \int_{\mc I_0^{**}}
        (\mathcal T_{t^2\slash (4u)}(x,y) -\mathbf T_{t^2\slash (4u)}(x,y) ) f(y)\, d\mu (y)\, du\Big|.
        \\
        \end{split}
        \end{equation*}
    Note that for $0<u<t^2\slash 4$ we have $0\leq \mathcal T_{t^2\slash (4u)}(x,y), \mathbf T_{t^2\slash (4u)}(x,y) \leq C$, see Lemma \ref{Lem:Gaussian}. Thus,
        \begin{equation*}
        \begin{split}
        |\mathcal P_tf(x)-\mathbf P_tf(x)|&
        \leq C \sup_{0<s<1} |(\mathcal T_{s} -\mathbf T_{s} ) f(x)| \int_0^\infty \frac{e^{-u}}{\sqrt u}
        \, du\\
        & \ \ +C  \int_0^{t^2\slash 4}\frac{e^{-u}}{\sqrt u}
        \int_{\mc I_0^{**}}
        |f(y)|\, d\mu (y)\, du
        \\
        &\leq  C \sup_{0<s<1} |(\mathcal T_{s} -\mathbf T_{s} ) f(x)| +C \| f\|_{L^1({\mc I_0^{**}}, \mu)}.
        \end{split}
        \end{equation*}
    The proof of Theorem \ref{theorem} is finished by applying Corollary \ref{heat}.
    \end{proof}

\section{Proof of Theorem \ref{thmA}}\label{sec4}

\subsection{Maximal function estimates}

In this subsection we provide some maximal function estimates that we shall use in the proof of Theorem \ref{thmA}.

We readily obtain from  Lemma \ref{Lem:SharpNatural} the following.
    \begin{lem}\label{prop:MinftyNatural}
    The operator $\mathcal{M}^\infty f(x):=\sup_{t>1}|\mathcal{P}_t f(x)|$ is bounded from $L^1((0,1), \mu)$ into itself.
    \end{lem}

    \begin{lem}\label{prop:Istarstar}
    Let $j\geq 1$. The maximal operator $\sup_{{\mu(\mc I^{**}_j)}<t<1}|\mathcal{P}_tf(x)|$ is bounded from $L^1(\mathcal{I}_{ j}^{**},\mu)$ into itself.
    \end{lem}

    \begin{proof}
    It is enough to see that, for $j\ge 1$,
    $$\int_{\mathcal{I}^{**}_{j}}\sup_{\mu(\mc I^{**}_j)<t<1}\mathcal{P}_t(x,y)\,d\mu(x)\leq C,$$
    for all $y\in\mathcal{I}^{**}_{j}$. The estimate is a direct consequence of Lemma \ref{Lem:SharpNatural} since $\mathcal{P}_t(x,y)\leq Ct^{-1}$
    for all $x,y\in\mathcal{I}_j^{**}$, $j\geq1$.
    \end{proof}

    \begin{lem}\label{prop:M0Natural}
    Let $\mathcal{M}^0f(x)=\sup_{0<t<1} |\mathcal{P}_t f(x)|$. If $f$ is in $L^1(\mathcal{I}_{ j}^{*},\mu)$, $j\ge0$, then
        \begin{equation}\label{bgf}
        \int_0^1\chi_{(\mathcal{I}_j^{**})^c}(x) \, \mathcal{M}^0f(x)\,d\mu(x)\le C\|f\|_{L^1(\mathcal{I}_j^*,\mu)},
        \end{equation}
    where $C$ is a constant independent of $f$ and $j$.
    \end{lem}

    \begin{proof} We consider two cases.

   \noindent
   {\bf Case 1: $j=0$.} We need to show that there exists a universal constant $C$ such that
   $$
  \chi_{\mathcal{I}_0^{*}}(y) \int_0^1\chi_{(\mathcal{I}_0^{**})^c}(x)\sup_{0<t<1}\mathcal{P}_t(x,y)\,d\mu(x)\leq C.
   $$

   \noindent
   Indeed, for $y\in\mathcal{I}_0^\ast$ and $x\in(\mathcal{I}_0^{**})^c\cap (0,1)$, we have $|x-y|>c$, $x\simeq 1$ and $(1-y)\simeq 1$. Under these assumptions, by Lemma \ref{Lem:SharpNatural}, $\mc P_t(x,y) \leq C$. The required estimate follows by noticing that $d\mu(x)\simeq dx$.

    \noindent{\bf Case 2. $j\ge1$.} Let $x\in(\mathcal{I}_j^{**})^c\cap (0,1)$. We analyze two subcases:

    \noindent{\bf Subcase 2.1: $1-2^{-j-1}<x<1$}. Note that, for $y\in\mathcal{I}_j^*$, $y\simeq 1$, $x\simeq 1$, $1-x<1-y\simeq  |\mathcal{I}_j|$ and $|x-y|\simeq |\mathcal{I}_j| \simeq \mu(\mc I_j)$.
    By Lemma \ref{Lem:SharpNatural} and the assumptions,
        $$\mathcal{M}^0f(x)\leq C \sup_{0<t<1}\int_{\mathcal{I}_j^*}\frac{t|\mathcal{I}_j|^2}{(t^2+|\mathcal{I}_j|^2)^2}
        |f(y)|\,d\mu(y)\leq C\frac{1}{2^{-j}}\|f\|_{L^1(\mathcal{I}_j^*,\mu)}.$$

    \noindent{\bf Subcase 2.2: $0<x<1-2^{-j}$}. Observe that for $y\in\mathcal{I}_j^*$, $y\simeq 1$, $|x-y|\simeq 1-x\ge1-y\simeq  |\mathcal{I}_j|\simeq 2^{-j}$. Again by Lemma \ref{Lem:SharpNatural}, and an analogous reasoning as in the previous subcase, we get
        $$\mathcal{M}^0f(x)\leq C \int_{\mathcal{I}_j^*}\sup_{0<t<1}\frac{t(1-x)2^{-j}}{(t^2+(1-x)^2)^2}|f(y)|\,d\mu(y)
        \leq C \frac{2^{-j}}{(1-x)^{2}} \|f\|_{L^1(\mathcal{I}_j^*,\mu)}.$$
By plugging both estimates into the integral \eqref{bgf}, we obtain the desired result for $j\geq 1$.
    \end{proof}

Now we make use of Theorem \ref{th:Uchiyama}. Consider the space of homogeneous type $(\mathcal{I}_{j}^{**},d_\mu,\mu)$, where
        $$d_\mu(x,y)= \Big|\int_x^y \, d\mu(z)\Big|,\qquad x,y\in\mathcal{I}_j^{**}.$$
        For $j\geq 1$, it is clear that $d_{\mu}$ and the usual distance $|\cdot|$ are equivalent on $\mc I^{**}_j$, but this is not the case for $\mc I^{**}_0$.

Let us begin with the case $j\ge1$. Observe that $\mu(\mc I^{**}_j) \simeq 2^{-j}$. Set $\mathcal{K}_{j}(t,x,y)=\mathcal{P}_t(x,y)$, for $0<t<\mu(\mc I_j^{**})$, $x,y\in\mathcal{I}_{j}^{**}$. Obviously,
    \begin{equation}\label{eq:identity1}
    \sup_{0<t<\mu(\mc I^{**}_j)}\left|\int_{\mathcal{I}_{ j}^{**}}\mathcal{K}_{j}(t,x,y)f(y)\,d\mu(y)\right|=\sup_{0<t<\mu(\mc I^{**}_j)}|\mathcal{P}_tf(x)|,
    \end{equation}
    for $x\in\mathcal{I}_{j}^{**}$.
In Lemma \ref{Prop:UchiNonLebesgue} we check the assumptions of Theorem \ref{th:Uchiyama} on each $\mathcal{I}_j^{**}$, see \eqref{UchiRep1}--\eqref{UchiRep3}.

    \begin{lem}\label{Prop:UchiNonLebesgue}
    There exists a constant $A>0$, independent of $j\geq 1$ such that for $0<t<\mu(\mc I^{**}_j)$ and $x,y, z\in \mc I^{**}_j$ we have that:
        \begin{enumerate}[(i)]
            \item $\mathcal{K}_{j}(t,x,x)>\frac{1}{At}$;
            \item $0\leq \mathcal{K}_{j}(t,x,y)\leq \frac{A}{t}\left(1+\frac{|x-y|}{t}\right)^{-2}$;
            \item if $|y-z|\leq\frac{1}{4A}(t+|x-y|)$, then
                $$|\mathcal{K}_{j}(t,x,y)-\mathcal{K}_{j}(t,x,z)|\leq A \frac{|y-z|}{t^2}\left(1+\frac{|x-y|}{t}\right)^{-2}.$$
        \end{enumerate}
    \end{lem}

    \begin{proof}
    \textbf{\textit{(i)} and \textit{(ii)}.} Since $1-x\simeq 1-y\simeq  2^{-j}$ and $x\simeq y \simeq 1$, we apply directly Lemma \ref{Lem:SharpNatural}.

    \noindent\textbf{\textit{(iii)}.} By the assumption we have that $t+|x-y| \simeq t+|x-z|$ and, by \textit{(ii)},
        \begin{equation*}\label{easy0}
        \mathcal{K}_{j}(t,x,y)+\mathcal{K}_{j}(t,x,z)\leq \frac{C}{t}\left(1+\frac{|x-y|}{t}\right)^{-2}.
        \end{equation*}
    So ${(iii)}$ is proved when $2|y-z|\geq t$. Assume $2|y-z|<t$. By the mean-value theorem,
        $$|\mathcal{K}_{j}(t,x,y)-\mathcal{K}_{j}(t,x,z)|=|y-z|\left|\frac{\partial}{\partial y}\mathcal{P}_t(x,y)\Big|_{y=\xi}\right|,$$
    for some $\xi$ between $y$ and $z$. Recall that $\mc P_t(x,y)= \mc P_t (y,x)$. From Lemma \ref{PropositionSixNonLeb} we have that
        $$|\mathcal{K}_{j}(t,x,y)-\mathcal{K}_{j}(t,x,z)| \le C\frac{|y-z|}{t^2} \(1+\frac{|x-\xi|}{t}\)^{-2}.$$
    When $|x-y|\leq t$ the conclusion follows immediately. In the opposite case we have $|x-y|>t>2|y-z|$, which implies $|x-\xi|\simeq |x-y|$ and $(iii)$ is proved.
    \end{proof}

    \begin{cor}\label{th:UchiNonLeb}
    Let $H^1_{\mathrm{at}}(\mathcal{I}_j^{**})$ be the atomic Hardy space defined as in Section \ref{sec2} for the space of homogeneous type $(\mathcal{I}_{j}^{**},d_\mu,\mu)$, and $j\ge1$. Assume that $f\in L^1{(\mc I_j^{**},\mu)}$. Then
    $$
    \|f\|_{H^1_{\mathrm{at}}(\mathcal{I}_j^{**})}\simeq \Big\|\sup_{0<t<\mu(\mc I_j^{**})}|\mathcal{P}_tf(x)|\Big\|_{L^1(\mathcal{I}_j^{**}, \mu)}.
    $$
    \end{cor}

    \begin{proof}
    The result follows from \eqref{eq:identity1} and Lemma \ref{Prop:UchiNonLebesgue} by applying Theorem \ref{th:Uchiyama}.
    \end{proof}

Our next goal is to obtain Corollary \ref{th:UchiNonLeb} also for $j=0$. This will follow from Theorem~\ref{theorem} and the characterization of the local Hardy space $h^1_{\mathbf L}$ related to the Bessel operator $\mathbf L$ that we state in Proposition \ref{ppprrr} below. It is worth to mention that the space $h^1_{\mathbf L}$ was described by means of the local Riesz transforms in \cite[Theorem 2.11]{Preisner}. The characterization with maximal functions is another consequence of \cite{Uchiyama} for which the key estimates were obtained in \cite{Betancor-Dziubanski-Torrea}. However, for the sake of completeness, we provide a sketch of the proof. Recall that we define the atomic Hardy space $H^1_{\mathrm{at}}(\mathcal{I}_0^{**})$ on the space of homogeneous type $(\mathcal{I}_{0}^{**},d_\mu,\mu)$ as in Section \ref{sec2}.
    \begin{prop}\label{ppprrr}
    For a function $f\in L^1{(\mathcal{I}_0^{**}, \mu)}$ the following holds
    $$
    \|f\|_{H^1_{\mathrm{at}}(\mathcal{I}_0^{**})}\simeq \Big\|\sup_{0<t<1}|\mathbf P_tf(x)|\Big\|_{L^1{(\mathcal{I}_0^{**},\, \mu)}}.
    $$
    \end{prop}

    \begin{proof}

    The proof is based on Theorem \ref{th:Uchiyama} and estimates obtained in \cite[Theorem 2.7]{Betancor-Dziubanski-Torrea}. We consider the space of homogeneous type $(\mc I_0^{**}, d_\mu, \mu)$. Observe that $\mu(\mc I_0^{**}) \simeq 1$. For $x,y\in \mc I_0^{**}$ and $r<\mu(\mc I_0^{**})$, we define (see \cite{Betancor-Dziubanski-Torrea})
        \begin{equation*}
        t(x,r) =  \begin{cases}
        rx^{-2\nu-1}, \qquad \quad r\leq x^{2\nu+2},\\
        r^{\frac{1}{2\nu+2}}, \qquad \qquad \, r > x^{2\nu+2}.
        \end{cases}
        \end{equation*}
    Set
        $$K(r,x,y)=\mathbf P_{t(x,r)}(x,y).$$
    The kernel $K(r,x,y)$ satisfies the assumptions of Theorem \ref{th:Uchiyama}, namely \eqref{UchiRep1}, \eqref{UchiRep2}, \eqref{UchiRep3}. The proof of this can be found in \cite[Proposition 2.12]{Betancor-Dziubanski-Torrea}. Thus, by defining $K_rf(x)$ as in \eqref{eq:maximalUchi}, from Theorem \ref{th:Uchiyama} we deduce that
        \begin{equation}\label{dsd}
        \|f\|_{H^1_{\mathrm{at}}(\mathcal{I}_0^{**})}\simeq \Big\|\sup_{0<r<\mu (\mc I^{**}_0)}| K_r f(x)|\Big\|_{L^1{(\mathcal{I}_0^{**},\, \mu)}}.
        \end{equation}
    Since $|K(r,x,y)|\leq C/r$ we have that
        \begin{equation}\label{sds}
        \sup_{\mu(\mc I^{**}_0)<r<1} \Big| \int_0^1 K(r,x,y) \, d\mu(x)\Big|\leq C.
        \end{equation}
Notice that $0<t<1$ if and only if $0<r<1$. From this, \eqref{dsd} and \eqref{sds}, for a function $f\in H^1_{\mathrm{at}}(\mathcal{I}_0^{**})$,
        \begin{equation*}
        \begin{split}
        \Big\|\sup_{0<t<1}| \mathbf P_t f(x)|\Big\|_{L^1{(\mathcal{I}_0^{**},\, \mu)}}  &\leq \Big\|\sup_{0<r<\mu(\mc I_0^{**})}| K_r f(x)|\Big\|_{L^1{(\mathcal{I}_0^{**},\, \mu)}} +\Big\|\sup_{\mu(\mc I^{**}_0)<r<1}| K_r f(x)|\Big\|_{L^1{(\mathcal{I}_0^{**},\, \mu)}} \\
        &\leq C \|f\|_{H^1_{\mathrm{at}}(\mathcal{I}_0^{**})} + C \|f\|_{L^1(\mathcal{I}_0^{**},\, \mu)}\leq C \|f\|_{H^1_{\mathrm{at}}(\mathcal{I}_0^{**})}.
        \end{split}
        \end{equation*}
    Finally, for a function $f$ such that $\Big\|\sup_{0<t<1}| \mathbf P_t f(x)|\Big\|_{L^1{(\mathcal{I}_0^{**},\, \mu)}}$ is finite, we get
        \begin{equation*}
        \|f\|_{H^1_{\mathrm{at}}(\mathcal{I}_0^{**})} \leq C \Big\|\sup_{0<r<\mu(I^{**}_0)}| K_r f(x)|\Big\|_{L^1{(\mathcal{I}_0^{**},\, \mu)}} \leq C \Big\|\sup_{0<t<1}| \mathbf P_t f(x)|\Big\|_{L^1{(\mathcal{I}_0^{**},\, \mu)}}
        \end{equation*}
    \end{proof}

Directly from Theorem \ref{theorem} and Proposition \ref{ppprrr} we obtain the following result.

    \begin{cor}\label{pprr} Let $H^1_{\mathrm{at}}(\mathcal{I}_0^{**})$ be the atomic Hardy space defined as in Section \ref{sec2} for the space of homogeneous type $(\mathcal{I}_{0}^{**},d_\mu,\mu)$. Assume that $f\in L^1{(\mc I_0^{**},\mu)}$. Then
    $$
    \|f\|_{H^1_{\mathrm{at}}(\mathcal{I}_0^{**})}\simeq \big\|\mathcal{M}^0f(x)\big\|_{L^1(\mathcal{I}_0^{**}, \mu)}.
    $$
    \end{cor}

\noindent\textbf{Partition of unitity.} For every $j\ge 0$, let $\eta_j\in C^\infty(0,1)$ be such that
    \begin{equation}\label{partition}
    \supp\, \eta_j\subseteq \mathcal{I}_j^{*}, \quad 0\leq\eta_j\leq 1,  \quad \Big|\frac{d}{dx}\eta_j(x)\Big|\leq C2^j,  \quad  \text{and}\quad \sum_{j=0}^{\infty}\eta_j(x)=1.
    \end{equation}

    The following lemma takes ideas from \cite{DDZZ}.
    \begin{lem}\label{lem:commutators}
    For $x,y\in (0,1)$ define
    $$
    \mathcal{V}(x,y)=\sum_{j=0}^{\infty}\sup_{0<t<\mu(\mc I^{**}_j)}|(\eta_j(x)-\eta_j(y))\mathcal{P}_t(x,y)|.
    $$
    Then, there exists a constant $C$ such that, for every $y>0$,
    \begin{equation}\label{dcv}
    \int_0^1\mathcal{V}(x,y)\,d\mu(x)\le C.
    \end{equation}
    \end{lem}

    \begin{proof}
    Recall that $\mu(\mc I^{**}_j) \simeq 2^{-j}$. Fix $y\in(0,1)$. Let $j_0$ be such that $y\in \mathcal{I}_{j_0}$. Throughout the proof of the lemma, we will use Lemma \ref{Lem:SharpNatural} repeatedly without further mention.

 \noindent\textbf{Case 1. $j_0\neq0$}. In this case $y>1/2$.

    \noindent\textbf{Subcase 1.1.} We first assume that $\mathcal{I}_j$ and $\mathcal{I}_{j_0}$ are separated enough, that is, $|j-j_0|\ge 2$. Then, $|\eta_j(x)-\eta_j(y)|\mathcal{P}_t(x,y)=\eta_j(x)\mathcal{P}_t(x,y)$.
    Notice that if $\eta_j(x)\neq 0$, then $|x-y|\simeq 2^{-j_0}+2^{-j}$. Thus, for $t<\mu(\mc I^{**}_j)$, we obtain
        $$        \eta_j(x)\mathcal{P}_t(x,y)\le \eta_j(x)\left(\frac{2^{-j}2^{-j_0}}{t^2+2^{-2j}+2^{-2j_0}}\right)\left(\frac{t}{t^2+2^{-2j_0}+2^{-2j}}\right)\le \eta_j(x)\frac{2^{-2j}2^{-j_0}}{2^{-4j}+2^{-4j_0}},       $$
    hence
        \begin{align*}
        \int_0^1\sum_{|j-j_0|\ge2}\sup_{0<t<\mu(\mc I_j^{**})}\eta_j(x)\mathcal{P}_t(x,y)\,d\mu(x)&\le \sum_{|j-j_0|\ge2}\int_0^1\eta_j(x)\frac{2^{-2j}2^{-j_0}}{2^{-4j}+2^{-4j_0}}\,d\mu(x)\\
        &\le \sum_{ j=0}^{\infty}\frac{2^{-3j}2^{-j_0}}{2^{-4j}+2^{-4j_0}}\le C.
        \end{align*}

    \noindent\textbf{Subcase 1.2.}
    Assume that $|j-j_0|\leq 1$. If $x\not \in \mathcal{I}^*_{j_0-1}\cup\mathcal{I}^*_{j_0}\cup\mathcal{I}^*_{j_0+1}$, then $\eta_j(x)=0$. Hence
    \begin{equation*}
    \int_{(\mathcal{I}^*_{j_0-1}\cup\mathcal{I}^*_{j_0}\cup\mathcal{I}^*_{j_0+1})^c}\sum_{|j-j_0|\leq 1}\sup_{0<t<\mu(\mc I^{**}_j)}|(\eta_j(x)-\eta_j(y))\mathcal{P}_t(x,y)| \leq C
    \end{equation*}
    by the proof of Lemma \ref{prop:M0Natural}.

    When $x\in \mathcal{I}^*_{j_0-1}\cup\mathcal{I}^*_{j_0}\cup\mathcal{I}^*_{j_0+1}$, we have $(1-x)\simeq (1-y)$ and $|x-y|\leq C 2^{-j}$. Then, for $0<t<\mu(\mc I^{**}_j)$, by the mean-value theorem we get
        \begin{align*}
        |\eta_j(x)-\eta_j(y)|\mathcal{P}_t(x,y)&\le C2^j|x-y|\mathcal{P}_t(x,y)\chi_{[\mathcal{I}_{j_0-1}^*\cup\mathcal{I}_{j_0}^*\cup\mathcal{I}_{j_0+1}^*]}(x)\\
        &\le C2^j\frac{t|x-y|}{t^2+(x-y)^2}\chi_{[\mathcal{I}_{j_0-1}^*\cup\mathcal{I}_{j_0}^*\cup\mathcal{I}_{j_0+1}^*]}(x)\\
        &\le C2^j\chi_{[\mathcal{I}_{j_0-1}^*\cup\mathcal{I}_{j_0}^*\cup\mathcal{I}_{j_0+1}^*]}(x).
        \end{align*}
    Thus
    $$\int_0^1\sum_{|j-j_0|\le1}\sup_{0<t<\mu(\mc I^{**}_j)}|\eta_j(x)-\eta_j(y)|\mathcal{P}_t(x,y)\,d\mu(x)\le C+C \sum_{|j-j_0|\le1}2^{-j_0}2^j=  C.    $$

   \noindent \textbf{Case 2. $j_0=0$}. Assume that $j\ge2$. Then $\eta_j(y)=0$ and, when $\eta_j(x)\not =0$, then $|x-y|>c$ and $x\simeq 1$. Hence
        \begin{equation}\label{gfgf1}
        \sup_{0<t<\mu(\mc I^{**}_j)}|\eta_j(x)-\eta_j(y)|\mathcal{P}_t(x,y)\leq C 2^{-j}\eta_j(x).
        \end{equation}
For $j=0,1$, we have
    \begin{equation}\label{gfgf2}
    \sup_{0<t<1}|\eta_j(x)-\eta_j(y)|\mathcal{P}_t(x,y)\le C \sup_{0<t<1}\frac{t|x-y|}{t^2+(x-y)^2}\le C.
    \end{equation}
Combining \eqref{gfgf1} and \eqref{gfgf2} we get \eqref{dcv}, when $y \in \mc I_0$.
        \end{proof}

\subsection{Proof of Theorem \ref{thmA}}

We begin with the proof of the first inequality in \eqref{th:A}. Let $a$ be an $\mathcal{L}$-atom, see Subsection \ref{subsection L redondo}. We prove that $\|\mathcal{M}a\|_{L^1{((0,1),\,\mu)}}\leq~C$, where $C$ is independent of $a$. Consider three cases:

\noindent{\bf Case 1:} Assume first that $\supp \, a\subseteq \mathcal{I}_0^{*}$. Then, either $a$ satisfies the cancellation condition and, in particular, it is an $H^1(\mc I_0^{**})$-atom (see Section \ref{sec2}), or $a=\mu(\mc  I_0)^{-1} \chi_{\mc I_0}(x)$. In the latter case, $a=\lambda_1 a_1+a_2$, where $a_1,a_2$ are $H^1(\mc I_0^{**})$-atoms and $|\lambda_1|\leq C$. Indeed,
    \begin{equation}\label{two_atoms}
    a=\mu(\mc  I_0)^{-1} \chi_{\mc I_0}(x)= \(\mu(\mc  I_0)^{-1} \chi_{\mc I_0}(x) -\mu(\mc I^{**}_0)^{-1} \chi_{\mc I^{**}_0}(x)\) + \mu(\mc I^{**}_0)^{-1} \chi_{\mc I^{**}_0}(x)=:\lambda_1 a_1+a_2.
    \end{equation}
Observe that
    \begin{equation*}
    \begin{split}
    \int_0^1\mathcal{M}a(x)\,d\mu(x)&\le \int_0^1\mathcal{M}^{\infty}a(x)\,d\mu(x)+\int_0^1 \chi_{(\mathcal{I}_0^{**})^c}(x)\mathcal{M}^0a(x)\,d\mu(x)\\
    &\quad+\int_{\mathcal{I}_0^{**}}  \mathcal{M}^0 a(x)\,d\mu(x).
    \end{split}
    \end{equation*}
By applying Lemmas \ref{prop:MinftyNatural} and \ref{prop:M0Natural}, and Corollary \ref{pprr}, we obtain that
these quantities are bounded by constants independent of $a$.

\noindent{\bf Case 2:} Suppose now that $\supp \, a\subseteq \mathcal{I}_j^{*}$ for $j\geq 1$. Exactly as in the previous case $a$ is an $H^1(\mc I_j^{**})$-atom or $a =\lambda_1 a_1+a_2$, where $a_1,a_2$ are $H^1(\mc I_j^{**})$-atoms and $|\lambda_1|\leq C$. Then
    \begin{align*}
    \int_0^1\mathcal{M}a(x)\,d\mu(x)&\le \int_0^1\mathcal{M}^{\infty}a(x)\,d\mu(x)+\int_0^1 \chi_{(\mathcal{I}_j^{**})^c}(x)\mathcal{M}^0a(x)\,d\mu(x)\\
    &\quad+\int_{\mathcal{I}_j^{**}} \sup_{\mu(\mc I^{**}_j)<t<1} |\mathcal P_t a(x)|\,d\mu(x)+\int_{\mathcal{I}_j^{**}} \sup_{0<t<\mu(\mc I^{**}_j)} |\mathcal P_t a(x)|\,d\mu(x).
    \end{align*}
By using Lemmas \ref{prop:MinftyNatural}, \ref{prop:Istarstar}, \ref{prop:M0Natural}, and Corollary \ref{th:UchiNonLeb}, the right-hand side is bounded by a constant independent of $a$ and $j$.

    \begin{rem}\label{ree1}
    If a function $b$ is such that $\supp \, b \subseteq \mc I_j^{*}$ and $\|b\|_\8 \leq C \mu (\mc I_j^{*})^{-1}$ for some $j\geq 0$, then
        $$\|\mc M b\|_{L^1((0,1),\,\mu)} \leq C.$$
    Indeed, set $b=b_1+b_2$, where
        $$b_1(x) = b(x) -  \frac{\int b \, d\mu}{\mu(\mc I_j)} \chi_{\mc I_j}(x), \qquad b_2(x) = \frac{\int b \, d\mu}{\mu(\mc I_j)} \chi_{\mc I_j}(x). $$
        Now, the claim follows from Case 1 ($j=0$) or Case 2 ($j\geq 1$) since $C^{-1} b_1$ and $C^{-1} b_2$ are $\mc L$-atoms with support contained in $\mc I_j^{*}$ and $C$ is some universal constant.
    \end{rem}
    Now we complete the remaining case.

\noindent{\bf Case 3:} Assume now that there is no $j$ such that $\supp\, a\subseteq \mathcal{I}_j^{*}$. Fix an interval $I$ such that $\supp\, a\subseteq I$ and $\|a\|_{\infty}\le\mu(I)^{-1}$ and fix the largest $N\in \NN\cup\{0\}$ and the smallest $M\in\NN \cup\{\infty\}$  such that $I\subseteq \bigcup_{j=N}^M \mathcal{I}_j^{*}$. Note that from the assumptions we have $N<M$ and $\mu(I)\simeq 2^{-N}$. Set
    $$    a=\sum_{j=N}^M2^{N-j}b_j ,   $$
where $b_j=a\cdot 2^{j-N}\chi_{\mathcal{I}_j}$.  Since $\|b_j\|_\8 \leq C 2^j\simeq C \mu(\mc I_j^*)$ and $\supp \, b_j \subseteq \mc I_j^*$, from Remark \ref{ree1} we deduce that $\|\mathcal{M}b_j\|_{L^1((0,1), \,\mu)}\le C$ and, consequently, $\|\mathcal{M}a\|_{L^1((0,1), \,\mu)}\le C$.

These three cases finish the first inequality in \eqref{th:A}. Let us now turn to proof of the converse, namely, the second inequality in \eqref{th:A}. Let $f\in H_{\mathcal{L}}^1$. We show that $f$ admits a suitable atomic decomposition. Recall the partition of unity $\eta_j$ in \eqref{partition}. We have that $f=\sum_{j=0}^{\infty}\eta_jf$ with $\supp\,\eta_j\subseteq \mathcal{I}_j^{*}$.
Observe that, by Lemma \ref{lem:commutators},
    \begin{equation}\label{eq:sum}
    \begin{split}
    \sum_{j=0}^{\infty}\Big\|\sup_{0<t<\mu(\mc I_j^{**})}|\mathcal{P}_t(\eta_jf)|\Big\|_{L^1(\mathcal{I}_j^{**},\mu)}&\leq
    \sum_{j=0}^{\infty}\Big\|\sup_{0<t<\mu(\mc I_j^{**})}|\mathcal{P}_t(\eta_jf)-\eta_j\mathcal{P}_t(f)|\Big\|_{L^1(\mathcal{I}_j^{**},\mu)}\\
    &\quad+\sum_{j=0}^{\infty}\Big\|\eta_j\sup_{0<t<\infty}|\mathcal{P}_t(f)|\Big\|_{L^1((0,1), \,\mu)}\\
    &\leq C \|f\|_{L^1((0,1), \,\mu)}+\|\mc M f \|_{L^1((0,1), \,\mu)}\le C\|f\|_{H^1_{\mathcal{L}}},
    \end{split}
    \end{equation}
    where $C$ is independent of $f$.
By Corollaries \ref{th:UchiNonLeb} and \ref{pprr}, for each $j\ge0$, $\eta_jf$ has an atomic decomposition into $H^1(\mc I_j^{**})$-atoms, so that
    $$    f(x)=\sum_{j=0}^{\infty}\eta_j(x)f(x)=\sum_{j=0}^{\infty}\sum_{k=1}^{\infty}\lambda_{j,k}a_{j,k}(x),$$
where $a_{j,k}$ are $H^1_{\mathrm{at}}(\mathcal{I}_j^{**})$-atoms, and
    $$\sum_{j=0}^{\infty}\sum_{k=1}^{\infty}|\lambda_{j,k}|\le C \sum_{j=0}^{\infty}\Big\|\sup_{0<t<\mu(\mc I_j^{**})}|\mathcal{P}_t(\eta_jf)|\Big\|_{L^1(\mathcal{I}_j^{**},\mu)}\le C \|f\|_{H^1_{\mathcal{L}}}.$$
Finally, for every $j,k$, either the function $C^{-1} a_{j,k}$ is an $\mathcal{L}$-atom, or $a_{j,k} = \mu(\mc I_j^{**})^{-1}\chi_{\mc I_j^{**}}$. In the latter case, $C^{-1} a_{j,k} $ is the sum of two $\mc L$-atoms, see \eqref{two_atoms} for a decomposition of $a_{j,k}$.

\section{Proof of Theorem \ref{thmB}}\label{sec5}

In the first part of this section similar results to those contained in Section~\ref{sec4} are stated. Observe that if $j\geq 1$ and $x\in  \mc J_j=\mc I_j $, then $x\simeq 1$ and the measure $\mu$ behaves like the Lebesgue measure. Therefore, when $j\geq 1$, the analysis of $P_t$ on $\mc J_j$ is almost identical to that of $\mc P_t$ on $\mc I_j$, see \eqref{Poisson kernel}. On $\mc J_j$ for $j\le-1$ we proceed similarly to the case $\mc J_{-j}$. Since most of the arguments are parallel to those of Section \ref{sec4}, we present only sketches of the proofs.

The first lemma follows directly from Lemma \ref{Lem:SharpLebesgue}.
    \begin{lem}\label{prop:MinftyLebesgue}
    The operator $M^\infty g(x):=\sup_{t>1}|P_tg(x)|$ is bounded from $L^1(0,1)$ into itself.
    \end{lem}

    \begin{lem}\label{prop:Jstarstar}
    Let $j\in \ZZ^*\setminus \{0\}$. The maximal operator $\sup_{|\mc J_j^{**}|<t<1}|P_tg(x)|$ is bounded from $L^1(\mathcal{J}_{ j}^{**})$ into itself.
    \end{lem}

    \begin{proof}
    Observe that $|\mc J_j^{**}|\simeq 2^{-|j|}$ and, for $|\mc J_j^{**}|<t<1$, by Lemma \ref{Lem:SharpLebesgue},
        $$\int_{\mathcal{J}^{\ast\ast}_{j}}\sup_{{|\mc J_j^{**}|}<t<1}P_t(x,y)\,dx\leq \int_{\mathcal{J}^{\ast\ast}_{j}}\sup_{{|\mc J_j^{**}|}<t<1}t^{-1}\,dx\leq C.$$
    \end{proof}

    \begin{lem}\label{prop:M0Lebesgue}
    Let $M^0g(x)=\sup_{0<t<1}|P_tg(x)|$, where $g$ is supported in $\mathcal{J}_{j}^{*}$, $j\in\ZZ^*$. Then, for a constant $C$ independent of $g$ and $j$,
        \begin{equation}\label{poi}
        \int_0^1 \chi_{(\mathcal{J}_{j}^{\ast\ast})^c}(x) \, M^0g(x)\,dx\le C\|g\|_{L^1(\mathcal{J}_{j}^\ast)}.
        \end{equation}
       \end{lem}

    \begin{proof}
    If $j\ge1$ then we actually can deduce \eqref{poi} from Lemma \ref{prop:M0Natural}. When $j\le-1$ we can proceed similarly to the proof of Lemma \ref{prop:M0Natural} by using Lemma \ref{Lem:SharpLebesgue}. The details are omitted.
    \end{proof}

Analogously to Section \ref{sec4}, we consider the space of homogeneous type $(\mathcal{J}_{j}^{**},|\cdot|,dx)$, for $j\in \ZZ^*$. Set $K_{j}(t,x,y)=P_t(x,y)$, for $0<t<|\mc J_j^{**}|\simeq 2^{-|j|}$, $x,y\in\mathcal{J}_{ j}^{**}$. We clearly have
    \begin{equation}\label{eq:identity2}
    \sup_{0<t<|\mc J_j^{**}|}\left|\int_{\mathcal{J}_{j}^{**}}K_{j}(t,x,y)g(y)\,dy\right|=\sup_{0<t<|\mc J_j^{**}|}|P_tg(x)|.
    \end{equation}

    \begin{lem}\label{Prop:UchiLebesgue}
    There exists a constant $A$, independent of $j$, such that for $0<t<|\mc J^{**}_j|$ and $x,y,z\in\mathcal{J}_{j}^{**}$ we have:
        \begin{enumerate}[(i)]
            \item $K_{j}(t,x,x)>\frac{1}{At}$;
            \item $0\leq K_{j}(t,x,y)\leq \frac{A}{t}\left(1+\frac{|x-y|}{t}\right)^{-2}$;
            \item if $|y-z|\leq\frac{1}{4A}(t+|x-y|)$, then
                $$|K_{j}(t,x,y)-K_{ j}(t,x,z)|\leq A\frac{|y-z|}{t^2}\left(1+\frac{|x-y|}{t}\right)^{-2}.$$
        \end{enumerate}
    \end{lem}

    \begin{proof}
    We consider only the case $j\leq -1$. The proof for $j\geq 1$ is very similar.

    \noindent\textbf{\textit{(i).}} If $x\in \mc J_j^{**}$, then $x\simeq 2^{j}$ and $1-x \simeq 1$. Since $|\mc J^{**}_j|\simeq 2^{j}$, from Lemma \ref{Lem:SharpLebesgue},
    $$K_{j}(t,x,x)\simeq \frac{1}{t}\left(\frac{x^2}{t^2+x^2}\right)^{\nu+1/2}\geq \frac{1}{At}.$$

    \noindent\textbf{\textit{(ii)}.} The proof follows directly from Lemma \ref{Lem:SharpLebesgue}.

    \noindent\textbf{\textit{(iii)}.} We proceed analogously to the proof of Lemma \ref{Prop:UchiNonLebesgue}. We only consider the case $2|y-z|<t$.  By the mean-value theorem,
    $$|K_{j}(t,x,y)-K_{j}(t,x,z)|=|y-z|\left|\frac{\partial}{\partial y}P_t(x,y)\Big|_{y=\xi}\right|,$$
    for some $\xi$ between $y$ and $z$. Recall that $x,y,z\simeq 2^{j}$, so also $\xi\simeq2^{j}$. By Corollary \ref{corrrr} and Lemma \ref{Lem:SharpLebesgue},
    \begin{align*}
        |K_{j}(t,x,y)-K_{j}(t,x,z)| &\leq  C\frac{|y-z|}{t^2}\left(\(1+\frac{|x-\xi|}{t}\)^{-2}+\frac{t^2}{\xi}P_t(x,\xi)\right)\\
        &\leq C\frac{|y-z|}{t^2}\(1+\frac{|x-\xi|}{t}\)^{-2},
    \end{align*}
    where in the last inequality we have used that $t<C2^{j}\simeq \xi$.
    \end{proof}

    \begin{cor}\label{th:UchiLeb}
    For $j\in\ZZ^*$,
    $$
    \|g\|_{H^1_{\mathrm{at}} \,(\mathcal{J}_j^{**})}\simeq \Big\|\sup_{0<t<|\mc J_j^{**}|}|P_tg(x)|\Big\|_{L^1{(\mathcal{J}_j^{**})}}.
    $$
    \end{cor}

    \begin{proof}
    The result follows from \eqref{eq:identity2} and Lemma \ref{Prop:UchiLebesgue} by applying Theorem \ref{th:Uchiyama}.
    \end{proof}

\noindent\textbf{Partition of unity.} For $j\in \ZZ^*$, let $\wt \eta_j\in C^\infty(0,1)$ such that:
    $$\supp\, \wt \eta_j\subseteq \mathcal{J}_j^{*}, \quad 0\leq\wt \eta_j\leq 1, \quad \Big|\frac{d}{dx}\wt \eta_j(x)\Big|\leq C2^{|j|} \quad \text{and} \quad \sum_{j\in \ZZ^*}\wt \eta_j(x)=1.$$

    \begin{lem}\label{lem:commutatorsLebesgue}
    The integral kernel
        $$    V(x,y)=\sum_{j\in \ZZ^*}\sup_{0<t<|\mc J_j^{**}|}|(\wt \eta_j(x)-\wt \eta_j(y))P_t(x,y)|,    $$
    satisfies
    $$\sup_{y\in(0,1)}\int_0^1V(x,y)\,dx\le C.
    $$
    \end{lem}

    \begin{proof}
    The lemma can be proved analogously to Lemma \ref{lem:commutators}. The details are omitted.
    \end{proof}

\subsection{Sketch of the proof of Theorem \ref{thmB}}
Let $a$ be an $L$-atom with $\supp\, a\subseteq \mathcal{J}_j$. Then,
    \begin{align*}
    \int_0^1Ma(x)\,dx\leq &\int_0^1M^{\infty}a(x)\,dx+\int_0^1 \chi_{(\mathcal{J}_j^{**})^c}(x)M^0a(x)\,dx\\
    &+\int_{\mathcal{J}_j^{**}} \sup_{|\mc J_j^{**}|<t<1} | P_t a(x)|\,dx+\int_{\mathcal{J}_j^{**}} \sup_{0<t<|\mc J_j^{**}|} |P_t a(x)|\,dx.
    \end{align*}
By Lemmas \ref{prop:MinftyLebesgue}, \ref{prop:Jstarstar}, \ref{prop:M0Lebesgue}, and Corollary \ref{th:UchiLeb}, the terms are bounded by constants independent of $a$.

An argument completely analogous to the one used in Remark \ref{ree1} shows that there exists a constant $C$ such that for a function $b$ with $\supp \, b \subseteq \mc J_j^*$ and $\|b\|_\8 \leq C |\mc J_j^*|^{-1}$ for $j\in \ZZ^*$, we have
    \begin{equation}\label{ree2}
    \|M b\|_{L^1(0,1)} \leq C.
    \end{equation}

Suppose now that there is no $j$ for which $\supp\, a\subseteq \mathcal{J}_j^{*}$, then we take $J$ such that $\supp\, a\subseteq J$ and $\|a\|_{\infty}\le|J|^{-1}$ and we fix the largest $N_1\in \ZZ^*\cup \{-\infty\}$ and the smallest $N_2\in \ZZ^*\cup \{\infty\}$ that verify
    $$J\subseteq \bigcup_{j\in\{N_1,...,N_2\}\setminus\{0\}}\mathcal{J}_j^{*}.$$
Set $K=1$ when $N_1<0$ and $N_2 >0$ and $K= \min(|N_1|,|N_2|)$ in the opposite case. Then $|J|\simeq 2^{-K}$. We can write
    $$ a=\sum_{j\in\{N_1,...,N_2\}\setminus\{0\}} 2^{K-|j|}b_j, $$
where $b_j=a\cdot 2^{|j|-K}\chi_{\mathcal{J}_j}$. Observe that $\supp \, b_j \subseteq \mc J_j^*$ and $\|b_j\|_\8 \leq C2^{|j|}$. Then, by \eqref{ree2}, we get $\|Mb_j\|_{L^1(0,1)}\le C$ and consequently, $\|Ma\|_{L^1(0,1)}\le C$.

For the converse, we take a function $g\in H_{L}^1$. By using the partition of unity above, $g=\sum_{j\in \ZZ^*}\wt \eta_jg$. By an analogous argument to that used in \eqref{eq:sum} together with Lemma \ref{lem:commutatorsLebesgue}, we obtain
$$\sum_{j\in \ZZ^*}^{\infty}\|M_j(\wt \eta_jg)\|_{L^1(\mathcal{J}_j^{**})}\le C \|g\|_{H^1_L}.$$
Now, by Corollary \ref{th:UchiLeb}, for each $j\in \ZZ^*$, $\wt \eta_jg$ has an atomic decomposition into atoms associated with $\mathcal{J}_j^{**}$, so that
    $$
    g(x)=\sum_{j\in \ZZ^*}\wt \eta_j(x)g(x)=\sum_{j\in \ZZ^*}\sum_{k=1}^{\infty} \lambda_{j,k}a_{j,k}(x),
    $$
where $a_{j,k}$ are $H^1_{\mathrm{at}}(\mathcal{J}_j^{**})$ atoms, and
    $$
    \sum_{j\in \ZZ^*}\sum_{k=1}^{\infty}| \lambda_{j,k}|\le C \sum_{j\in \ZZ^*}\|M_j(\wt \eta_jg)\|_{L^1(\mathcal{J}_j^{**})}\le C \|g\|_{H^1_L}.
    $$
The proof is finished, since every $C^{-1}a_{j,k}$ is an $L$-atom or it is the sum of two $L$-atoms, see the end of the proof of Theorem \ref{thmA}.

\medskip
\noindent{\bf Acknowledgments.} This research started when the fourth author was a researcher at Universidad de La Rioja, Spain, under grant COLABORA 2010/01 from Planes Riojanos de I+D+I, Spain. The second author is grateful to Departamento de Matem\'aticas y Computaci\'on of Universidad de La Rioja, Spain, for their hospitality. The third author is also grateful to Instytut Matematyczny of Uniwersytet Wroc\l awski, Poland, for their hospitality.

The authors are also grateful to \'Oscar Ciaurri, Adam Nowak and Jos\'e L. Torrea for useful discussions and comments.


\end{document}